\documentclass[a4paper,12pt]{article}

\usepackage[utf8]{inputenc} 
\usepackage[T1]{fontenc}    
\usepackage[english]{babel} 
\usepackage{a4wide}         
\usepackage{hyperref}
\usepackage{amsmath, amssymb, amsthm, amsfonts , stmaryrd} 
\usepackage{mathrsfs}  
\usepackage{bbold}     
\usepackage{bigints}   
\usepackage{upgreek}   
\usepackage{systeme}   
\usepackage{calc}      

\usepackage{graphicx}
\graphicspath{{./images/}} 
\usepackage[font=small,labelfont=bf]{caption} 
\usepackage{subcaption} 

\usepackage{hyperref}

\newtheoremstyle{custom_theorem} 
  {10pt} 
  {10pt} 
  {\itshape} 
  {} 
  {\bfseries} 
  {.} 
  { } 
  {} 

\newtheoremstyle{custom_remark} 
  {10pt} 
  {10pt} 
  {\normalfont} 
  {} 
  {\bfseries} 
  {:} 
  { } 
  {} 

\theoremstyle{custom_theorem}
\newtheorem{theorem}{Theorem}[section]
\newtheorem{proposition}[theorem]{Proposition}
\newtheorem{lemma}[theorem]{Lemma}

\newtheorem{definition}[theorem]{Definition}

\theoremstyle{custom_remark}
\newtheorem{remark}{Remark}[section]
\newtheorem{example}{Example}[section]

\numberwithin{equation}{section}

\DeclareMathOperator{\tr}{Tr}

\DeclareMathOperator{\supp}{supp}

\newcommand{\diam}{\mathrm{diam}}

\newcommand{\norm}[1]{\left\Vert #1 \right\Vert}
\renewcommand{\P}{\mathbb{P}}
\newcommand{\Z}{\mathbb{Z}}
\newcommand{\E}{\mathbb{E}}
\newcommand{\N}{\mathbb{N}}
\newcommand{\1}{\mathbb{1}}
\newcommand{\C}{\mathbb{C}}

\newcommand{\CC}{\mathcal{C}}
\newcommand{\M}{\mathcal{M}}

\newcommand{\A}{\mathcal{A}}
\newcommand{\B}{\mathcal{B}}
\newcommand{\F}{\mathcal{F}}

\newcommand{\dd}{\mathrm{d}}

\DeclareUnicodeCharacter{2212}{\ensuremath{-}}

\begin{document}

\author{Alexis Imbert \thanks{Universit\'{e} de Bordeaux, Institut de Math\'{e}matiques de Bordeaux, 351 Cours de la Lib\'{e}ration, 33400 Talence, France. Email : \href{mailto:alexis.imbert@math.u-bordeaux.fr}{alexis.imbert@math.u-bordeaux.fr}}}

\title{A random matrix approach to lamplighter groups}
\date{\today}
    \maketitle
    \begin{abstract}
        Let $\Lambda$ be a finitely generated abelian group and $\Gamma=\Z^{*d}$, we study the Cayley graph of the wreath product $G=\Lambda\wr\Gamma$ with natural set of generators and their inverse $S$. First, we establish a random matrix model $X_N=\sum_sX^{(s)}_N$ where the sum is indexed by the set $S$. As the size $N$ of the matrices goes to infinity, the traffic distribution of the $X^{(s)}_N$'s converges to that of the image of these generators in the reduced $C^*$-algebra of $G$. In particular, the spectral measure of $X_N$ converges toward that of the Cayley graph of $G$ with generators $S$. Moreover, in the case $\Gamma=\Z$, we establish a central limit theorem for linear statistics of this random matrix model.
        Then, we exhibit a formula for the asymptotic $R$-transform and derive the second-order distribution of the limit of $X_N$ in terms of its limiting first-order distribution.
    \end{abstract}
    
	\section{Introduction} 

Studying the spectra of finitely generated non commutative groups has been a challenging problem launched by Kesten in \cite{kesten_symmetric_1959}. Let $G$ be a group generated by $S\subset G$, a finite subset of generators with $S=S^{-1}:=\{s^{-1},\,s\in S\}$. We are interested in the spectral measure of the Markov operator $M\in \ell^{2}(G)$ associated with the simple random walk on $G$. There are few example of computation of the spectra of graphs and even fewer examples of computation of their spectral measures (see the survey \cite{mohar_survey_1989})

Let $\Gamma=\Z^{* d}$ be the free group with $d$ generators and $\Lambda$ be a finitely generated abelian group. The wreath product $G:=\Lambda\wr\Gamma$ is $(\bigoplus_{\Gamma}\Lambda)\rtimes \Gamma$ equipped with the action of $\Gamma$ on $\bigoplus_{\Gamma}\Lambda$ by a shift. In other words, an element $h\in G$ is a couple $h=(f,g)$ where $f:\Gamma\rightarrow \Lambda$ has finite support and $g\in\Gamma$. The multiplication is given by 
$$(f,g)\cdot(f',g')=(f\cdot(f'\circ g^{-1}),gg'),$$
where $f\circ g^{-1}(\gamma)=f(g^{-1}\gamma)$ .
Note that there are natural inclusions of $\Gamma$ and $\Lambda$ into $G$ through the maps $\lambda\rightarrow\hat{\lambda}=( \mathbf{1}_{\lambda},e_{\Gamma})$, where $ \mathbf{1}_{\lambda}(e_{\Gamma})=\lambda$ and $\mathbf{1}_{\lambda}(\gamma)=e_{\Lambda}$ if $\gamma\neq e_{\Gamma}$; and $\gamma\rightarrow\hat{\gamma}=(e,\gamma)$, where $e=e_{\oplus_{\Gamma}\Lambda}$.

The wreath product of two groups is studied in physics literature from the random Schrödinger operator point of view, see \cite{arras} and the references therein for a more physical approach to this problem. The standard example is that of $\Gamma=\Z$ and $\Lambda =\Z/2\Z$, and is called the lamplighter group. The name of the group comes from viewing it as an infinite sequence of lamps on $\Z$ which may be on or off (determined by $f$) and with a lamplighter being at position $g\in\Z$. We thus refer to $\Lambda$ as the lamp group and $\Gamma $ as the base group.

It has been shown in \cite{grigorchuk_lamplighter_2001} that for a specific choice of generators, the spectral measure of the lamplighter group is discrete and can be explicitly stated. Moreover, there has been some studies on the spectrum on the Markov operator $M_{P}$ associated to some symmetric distribution $P$ on the set of generators (see \cite{grigorchuk_spectra_2019}) and it shows that the nature of the spectral measure is heavily dependent on the choice of the generators. One can also look at \cite{Revelle2003,DicksSchick2002,KambitesSilvaSteinberg2005} for further reading. However, to the author's knowledge, little is known about the spectrum of the Markov operator for the natural set of generators.

Let $m\geq 0$, $r_2\geq0$ and $q_1,\cdots,q_t\geq 3$ be fixed throughout the paper so that $$\Lambda\cong \Z^{m}\times \left(\Z/2\Z\right)^{r_2}\times\Z/q_1 \Z\times \cdots\times \Z/q_t \Z,$$ where we distinguished $\Z/2\Z$ from the other finite abelian groups for reasons that will become clear later. Let $S_\Lambda$ be the natural set of generators of $\Lambda$ and their inverse, i.e. $(1,0,\cdots,0),(0,1,\cdots,0)$ and so on with their inverse. Moreover, let $S_\Gamma$ be the natural set of generators of $\Gamma=\Z^{* d}$ with their inverse. We consider throughout this paper the set of generators $S=\hat{S_\Gamma}\sqcup\hat{S_{\Lambda}}\subset G$.

The Markov operator $M:\ell^{2}(G)\rightarrow\ell^{2}(G)$ associated to the simple random walk on the Cayley graph of $G$ with generators $S$ is  
\begin{equation}
    M\phi(h)=\frac{1}{|S|}\sum_{s\in S}\phi(sh),
\end{equation}
where $\phi\in\ell^{2}(G)$ and $h\in G$.
As the operator $M$ is bounded ($\norm{M}\leq1$) and self-adjoint it has a spectral decomposition 
\begin{equation}
    M=\int_{-1}^{1}\lambda\dd E(\lambda),
\end{equation}
where $E$ is the spectral measure. This measure is defined on the Borel subsets of the interval $[-1; 1]$ and takes values in the projections on the Hilbert space $\ell^{2}(G)$. 
 
The spectral measure $\mu$ associated to this operator on the interval is defined as
\begin{equation}
    \mu(B)=\langle E(B)\delta_{e_G},\delta_{e_G}\rangle,
\end{equation}
where $B$ is a Borel subset of $[-1; 1]$ and $\delta_{e_G}\in\ell^{2}(G)$ is the function which equals 1 at the identity element and 0 elsewhere. 
    
Let $p_{n}(e_G,e_G)$ be the probability of return to the null element after n steps of the simple random walk. Then we have : 
\begin{equation}\label{self-returning paths}
    p_{n}(e_G,e_G)=\int_{-1}^{1}\lambda^{n}\dd\mu(\lambda).
\end{equation}

We exhibit a random matrix model whose spectral measure approaches $\mu$ as the size of the matrices go to infinity. This approach is quite common when studying the spectral measure of a graph (not necessarily the Cayley graph of some group) as the measure of cartesian, star, rooted and free products of graphs can be understood as different notions of independence in non-commutative probability (see \cite{accardi_decompositions_2006}, \cite{speicher_universal_1997}). Here, we investigate the link between wreath product of groups and the notion of traffic independence and freeness over the diagonal.

\paragraph{The random matrix model.} Let $N\geq 1$ be the size of all matrices to be mentioned. Let $V_1,\ldots,V_d$ be i.i.d. permutation matrices and let $D_1,\ldots,D_m,D'_1,\ldots,D'_{r_2},\Delta_1,\ldots\Delta_t$ be diagonal matrices whose entries are all independent and independent from the $V_i$'s. The entries of $D_1,\ldots,D_m$ are i.i.d. and follow a uniform law on the unit circle $\mathbb{T}$. The entries of the $D'_j$'s are i.i.d. and follow a uniform law on $\{\pm1\}$ and the entries of $\Delta_j$ follow a uniform law on $\mathbb{U}_{q_j}:=\{\exp(2ik\pi/q_j),\,1\leq k \leq q_j\}$. We define 
\begin{equation}\label{eq: matrix model}
    X_N:=\frac{1}{2d+2m+2t+r_2}\left(\sum_{i=1}^{d}(V_i+V_i^{*})+\sum_{j=1}^{m}(D_j+D_j^*)+\sum_{j=1}^{t}(\Delta_{j}+\Delta_j^*)+\sum_{r=1}^{r_2}D'_r\right).
\end{equation}
In the case of the lamplighter group, it reduces to $X_N=1/3(V+V^*+D')$. Note that $2d+2m+2t+r_2=|S|$, where $S$ is the natural set of generators and their inverse of $G$. The group $\Z/2\Z$ is distinguished since its generator is its own inverse, thus not adding an element to the set $S$.

We use the tools and methods developed in \cite{male_traffic_2020} to compute the moments of $X_N$. Theorem \ref{measure convergence} states that the empirical spectral measure of $X_{N}$ converges to the spectral measure of the Cayley graph of the lamplighter group $\mu$. 
\begin{theorem}\label{measure convergence}
    Let $\mu_N$ be the empirical spectral distribution of $X_N$, weakly, in probability, $\mu_N\rightarrow\mu$ as $N$ goes to infinity.
\end{theorem}

From Section 3 until the end of this paper, we focus on the case where $d=1$, or equivalently, where $\Gamma = \Z$. We show a central limit theorem on the linear statistics of $X_{N}$ using a method similar to that of \cite{benaych-georges_central_2014}
\begin{theorem}\label{TCL}
    Let $n,N\in\N$, and define
    \begin{equation}
		Z_{N}(n):=\frac{1}{\sqrt{N}}\left[\mathrm{Tr}(X_{N}^{n})-\E(\mathrm{Tr}(X_{N}^{n}))\right].
	\end{equation}
The family $(Z_{N}(n))_{n\in\N}$ converges to a Gaussian process as $N$ goes to infinity.
\end{theorem}

Theorem \ref{measure convergence} and Theorem \ref{TCL} are proven in Section 2 and 3 respectively.

 The second part of this article establishes two main results whose precise statements require the operator-valued framework developed in Sections 4 to 6 : Propositions \ref{analytic R transform} and \ref{prop: second-order distrib}. We state here the ideas and refer the reader to Sections 4 to 6 for a proper statement.
 
These three sections focus on establishing a limiting operator-valued $C^*$ probability space for this model. An operator-valued $R$-transform and the second-order distribution are then established. Operator-valued probability was introduced by Voiculescu in \cite{voiculescu_symmetries_1985} and Speicher showed in \cite{speicher_combinatorial_1998} that the combinatorial description of free cumulants also extends to operator-valued non-commutative probability spaces. In \cite{au_freeness_2021}, Au et al. showed that permutation invariant random matrices are asymptotically free over the diagonal, allowing us to use this kind of freeness in our model. Since lamplighter matrices do not converge to an operator-valued semi-circular variable, the $R$-transform is not linear as in the heavy-Wigner model (see \cite{bordenave2024largedeviationsmacroscopicobservables}). We prove in Section 5, Proposition \ref{analytic R transform} which provides an explicit formula for the asymptotic $R$ transform of the "limit" of $V_N+V_N^*$ in some operator-valued probability space that depends on the choice of the lamp group $\Gamma$.

     In the last section, we show that in this limiting space, we are able to express the second order distribution of two variables in terms of the first order of the product of the two, up to some symmetrization: see  Proposition \ref{prop: second-order distrib}.

    \section{Proof of Theorem \ref{measure convergence}}

In this section, we investigate the type of independence that links the matrices $V_,D,\Delta$ and $D'$, namely, traffic independence. We start by recalling a few definitions from \cite{male_traffic_2020} and \cite{male_limiting_2017} regarding traffic independance. We then show the convergence of the expectations of the moments of $\mu_{N}$ toward the moments of $\mu$. Finally, we show the convergence in probability.
 
\subsection{Traffic independence for random matrices}
    
	\begin{definition}
        Let $J$ be a set of indices and consider two families of formal variables $\mathbf{x}=(x_{j})_{j\in J}$ and $\mathbf{x^{*}}=(x_{j}^{*})_{j\in J}$.
	    \begin{enumerate}
	        \item A \emph{*-test graph} $T = (G, \gamma, \epsilon )$ in the variables $\mathbf{x}$ is a finite connected multi-digraph $G = (V , E, \mathrm{src}, \mathrm{tar})$ together with edge labels $\epsilon : E \rightarrow \{1,*\}$ and $\gamma : E \rightarrow J$. One can see the maps $\gamma$ and $\epsilon$ as indicating that an edge $e\in E$ is labeled $x_{\gamma(e)}^{\epsilon(e)}$. The maps $\mathrm{src}$, $\mathrm{tar} : E \rightarrow V$ specify the source, $\mathrm{src}(e)$ and target, $\mathrm{tar}(e)$ of each edge $e \in E$. 
            \item A \emph{*-graph monomial} $g=(T,v_{in},v_{out})$ is a *-test graph $T$ together with the data of two vertices  $v_{in}$ and $v_{out}$ in $V$. We refer to the roots $(v_{in}, v_{out}) \in V^{2}$ as the input and the output, respectively, though they need not be distinct. We denote by $\mathcal{G}\langle\mathbf{x},\mathbf{x^*}\rangle$ the set of all such *-graph monomials which can be extended to $\C\mathcal{G}\langle\mathbf{x},\mathbf{x^*}\rangle$ the space of finite linear complex combinations of *-graph monomials in variable $\mathbf{x}$ (graphs are considered up to isomorphisms of graphs preserving labeled and in/outputs).
	    \end{enumerate}
	\end{definition}
    For a family $\mathbf{A}_{N}=(A_{N,j})_{j\in J}$ of random matrices, we define the evaluation of a *-graph monomial $g$ in the family $\mathbf{A}_{N}$ via the formula,
    \begin{equation}
        g(\mathbf{A}_{N})(i,j):=\sum_{\substack{\phi:V\rightarrow[N] \\ \phi(v_{out})=i,\phi(v_{in})=j}}\prod_{e\in E}A^{\epsilon(e)}_{N,\gamma(e)}(\phi(\mathrm{tar}(e)),\phi(\mathrm{src}(e))).
    \end{equation}
    For convenience, we will often denote $\phi(e)$ for $(\phi(\mathrm{tar}(e)),\phi(\mathrm{src}(e)))$.
    When one studies random matrices in the framework of free probability, one is interested in the expectation of the normalized trace of any polynomials evaluated in a family of random matrix. In this framework, the \emph{traffic distribution} of a family of random matrices is the data of the expectation of the normalized trace of any *-graph monomial evaluated in this family.

    \begin{definition}
        The \emph{traffic distribution} of a family $\mathbf{A}_{N}$ of random matrices is the map
        \begin{equation}
            \Phi_{\mathbf{A}_{N}}:g\in\C\mathcal{G}\langle\mathbf{x},\mathbf{x^*}\rangle\rightarrow\E\left[\frac{1}{N}\tr\left(g(\mathbf{A}_{N})\right)\right].
        \end{equation}
        We say that the family $\mathbf{A}_{N}$ \emph{converges in traffic distribution} if $\Phi_{\mathbf{A}_{N}}$ converges point-wise.
    \end{definition}

    \begin{remark}
        Note that for $g\in\mathcal{G}\langle\mathbf{x},\mathbf{x^*}\rangle$ a *-graph monomial, $\Phi_{\mathbf{A}_{N}}(g)=\Phi_{\mathbf{A}_{N}}(\Delta(g))$, where $\Delta(g)$ is obtained from $g$ by identifying (gluing) the vertices $v_{in}$ and $v_{out}$. Moreover, since we take the normalized trace, this quantity does not depend on the position of the vertex $v_{in}=v_{out}$ in $\Delta(g)$ but only on the *-test graph $T$ defined by $\Delta(g)=(T,v_{in},v_{in})$.
    \end{remark}
    The remark above motivates the following definitions.
    \begin{definition}
    Let $T=(V,E,\gamma,\epsilon)$ be a *-test graph , and $\mathbf{A}_{N}=(A_{N,j})_{j\in J}$ a family of random matrices.
    \begin{enumerate}
        \item  The \emph{trace} of $T$ evaluated in $\mathbf{A}_{N}$ is,
        \begin{equation}
            \tr\left(T(\mathbf{A}_{N})\right):=\sum_{\phi:V\rightarrow[N]}\prod_{e\in E}A^{\epsilon(e)}_{N,\gamma(e)}(\phi(e)).
        \end{equation}
        Note that $\tau_{N}(T(\mathbf{A}_{N})):=\E\left(\frac{1}{N}\tr(T(\mathbf{A}_{N}))\right)=\Phi_{\mathbf{A}_{N}}(g)$, where $g$ is any *-graph monomial verifying $\Delta(g) = (T,v_{in},v_{in})$ for some $v_{in}\in V$.
        
        \item The \emph{injective trace} of $T$ evaluated in $\mathbf{A}_{N}$ is,
        \begin{equation}
            \mathrm{Tr^{0}}\left(T(\mathbf{A}_{N})\right):=\sum_{\substack{\phi:V\rightarrow[N]\\\phi \,\mathrm{injective}}}\prod_{e\in E}A^{\epsilon(e)}_{N,\gamma(e)}(\phi(e)).
        \end{equation}
        Similarly, we define $\tau_{N}^{0}(T(A_{N})):=\E\left(\frac{1}{N}\mathrm{Tr^{0}}(T(\mathbf{A}_{N}))\right)$.
    \end{enumerate}
    The trace and injective trace of a *-test graph are related one to another in a similar fashion as cumulants are related to moments. Indeed, let $T=(V,E,\gamma,\epsilon)$ be a *-test graph and $\pi\in\mathcal{P}(V)$ be a partition of the vertices of $V$, we define a new *-test graph $T^{\pi}$ obtained from $T$ by identifying (gluing) the vertices of $T$ in a same block of $\pi$. Then for any *-test graph $T$, and any family $\mathbf{A}_{N}$ of random matrices, one has
    \begin{equation}
         \tr\left(T(\mathbf{A}_{N})\right)=\sum_{\pi\in\mathcal{P}(V)} \mathrm{Tr^{0}}\left(T^{\pi}(\mathbf{A}_{N})\right).
    \end{equation}
    The same relation holds when replacing $\tr$ (resp. $\mathrm{Tr^{0}}$) by $\tau_{N}$ (resp. $\tau_{N}^{0}$). Both of these relations can be inverted using the Möbius transform and one can thus obtain an expression of the injective trace of a *-test graph $T$ in terms of the trace of the $T^{\pi}$'s.
    \end{definition}

    \begin{remark}
        As shown in \cite[Lemma 2.9]{male_traffic_2020}, for a given family $\mathbf{A}_{N}$ of random matrices, the convergence in traffic distribution as defined above is equivalent to the convergence of $\tau_{N}(T(\mathbf{A}_{N}))$ for all *-test graph $T$ which is in turn equivalent to the convergence of $\tau^{0}_{N}(T(\mathbf{A}_{N}))$ for all *-test graph $T$. It is often easier to compute explicitly the limit in terms of the normalized trace, hence its introduction.
    \end{remark}

    We now give the limiting traffic distribution of permutation matrices and diagonal matrices.

    \begin{example}\label{diagonal}
        Let $D$ be a diagonal matrix with i.i.d. entries of law $\nu$. Let us compute the normalized trace of a *-test graph $T=(V,E,\varepsilon)$ in $D$ and see when it vanishes. Note that we do not need to add $\gamma$ since we are considering a family of only one matrix. We have 
        \begin{align*}
            \E\left[\frac{1}{N}\tr^{0}(T(D_{N}))\right] &=\E\left[\frac{1}{N}\sum_{\substack{\phi:V\mapsto[N],\\\text{injective}}}\prod_{e=(v,w)\in E}D^{\varepsilon(e)}(\phi(w),\phi(v))\right]\\*
            &=\1_{|V|=1}\E\left[D(1,1)^{|E_1|}\overline{D(1,1)}^{|E_{*}|}\right],
        \end{align*}
        where $E_1:=\{e\in E,\,\varepsilon(e)=1\}$ and $E_*=E\setminus E_1$.
        When $\nu$ is a Rademacher law, $\tau_{N}^{0}(T(D_{N}))=\1_{|V|=1,\,|E|\text{ even}}$. When $\nu$ is the uniform law on $\mathbb{U}_q$ the set of $q$-th roots of the unity, then $\tau_{N}^{0}(T(D_{N}))=\1_{|V|=1,\,|E_1|-|E_*|=0\,\mathrm{mod} \,q}$. When $\nu$ is the uniform law on $\mathbb{T}$, then $\tau_{N}^{0}(T(D_{N}))=\1_{|V|=1,\,|E_1|-|E_*|=0}$.
    \end{example}

    \begin{example}\label{uniform}
    Let $V$ be a uniform permutation matrix, it is shown in \cite{male_traffic_2020} that $\tau_{N}^{0}(T(V))\rightarrow\1_{T\text{ is a directed line}}$. A directed line $T$ is a connected graph where the vertices $1,\cdots,K$ are linked by edges $(1,2),\cdots,(K-1,K)$ labeled $v$ with arbitrary multiplicity, and by edges  $(K,K-1),\cdots,(2,1)$ labeled $v^{*}$, also with arbitrary multiplicity. We implicitly denoted the edges $e=(v,w)$ where $v$ is the source ($src(e)=v$) and $w$ is the target ($tar(e)=w$).

    Note that, from this information, we can compute the limiting *-moments of $V$. Let $P$ be a *-monomial in $v,v^{*}$, $P$ can be seen as a directed path labeled by $v$ and $v^*$ respecting the order of the product. For instance, $P=v^{2}v^{*}v^{2}(v^{*})^{3}$ is represented by the graph-monomial in Figure \ref{fig:exemple1},
    \begin{figure}[ht] 
        \centering \includegraphics[width=0.5\textwidth]{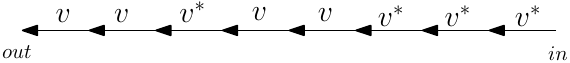}  
        \caption{The graph-monomial representation of $P$.}  
        \label{fig:exemple1} 
    \end{figure}
    
    \noindent where the input on the right-hand side and output on the left-hand side. Since for any graph-monomial $g$, $\tr(g)=\tr(\Delta(g))$, we only need to consider the graph-monomial of Figure \ref{fig:exemple2} in order to compute the asymptotic $*$-moments of $V_{N}$.
    \begin{figure}[ht] 
        \centering
        \includegraphics[scale=0.7]{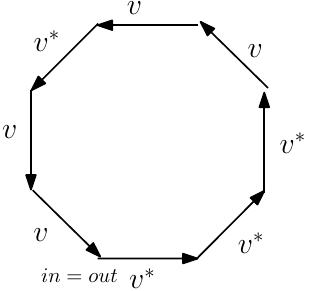} 
        \caption{The graph-monomial $\Delta(P)$.} 
        \label{fig:exemple2} 
    \end{figure}
    
    It is easy to see that if $P$ is balanced (i.e. it has as many $v$ labels as $v^{*}$ labels), then there is a unique partition $\pi$ such that $g_{P}^{\pi}$ is a directed line. Otherwise, no such partition exists. 
    \end{example}

    \begin{definition}[Graph of colored components]
        \begin{enumerate}
            \item Let $T$ be a *-test graph in the variables $\mathbf{x}=\mathbf{x_{1}}\cup\cdots\cup\mathbf{x_{p}}$, where the $\mathbf{x_{j}}$’s are families of pairwise disjoint variables (a variable appears at most in one family). A colored component of $T$ with respect to $\mathbf{x_{1}},\cdots,\mathbf{x_{p}}$ is a maximal connected sub-graph of $T$, with at least one edge, whose edges are labeled by variables in only one family among $\mathbf{x_{1}},\cdots,\mathbf{x_{p}}$. We denote by $\mathcal{CC}(T)$ the set of colored components of T with respect to $\mathbf{x_{1}},\cdots,\mathbf{x_{p}}$.
            \item The \emph{graph of colored components} of $T$ with respect to $\mathbf{x_{1}},\cdots,\mathbf{x_{p}}$, denoted $\mathcal{GCC}(T)$, is the following bipartite undirected graph.
            \begin{itemize}
                \item The first kind of vertices are the colored components $T_{1},\cdots, T_{K}$ of $\mathcal{CC}(T)$.
                \item The second kind of vertices are the vertices $v_1,\cdots,v_{L}$ of $T$ that belong to at least two graphs among $T_{1},\cdots, T_{K}$.
                \item There is an edge between $T_{i}$ and $v_j$ if $v_j$ is a vertex of $T_i$, $i=1,\cdots,K,\,j=1,\cdots, L$.
            \end{itemize}
        \end{enumerate}
    \end{definition}

    \noindent Recall that $\tau_{N}(T(\mathbf{A}_{N})):=\E\left(\frac{1}{N}\tr(T(\mathbf{A}_{N}))\right)$ and $\tau^{0}_{N}(T(\mathbf{A}_{N})):=\E\left(\frac{1}{N}\mathrm{Tr^{0}}(T(\mathbf{A}_{N}))\right)$.
    \begin{definition}
        We say that $\mathbf{A}_{N}^{(1)},\cdots,\mathbf{A}_{N}^{(L)}$ are \emph{asymptotically traffic independent} whenever $\mathbf{A}_{N}:=\mathbf{A}_{N}^{(1)}\cup\cdots\cup\mathbf{A}_{N}^{(L)}$ converges in traffic distribution and for any *-test graph $T=(V,E,\gamma,\epsilon)$,
        \begin{equation}\label{traffic freeness}
            \tau_{N}\left(T(\mathbf{A}_{1}^{(N)},\cdots,\mathbf{A}_{L}^{(N)})\right)\underset{N\rightarrow\infty}{\longrightarrow}\sum_{\substack{\pi\in\mathcal{P}(V) \\ \mathrm{s.t.}\mathcal{GCC}(T^{\pi}) \\ \text{is a tree}}}\prod_{S\in\mathcal{CC}(T^{\pi})}\lim_{N\rightarrow\infty}\tau_{N}^{0}\left(S(\mathbf{A}_{N}^{(\ell(S))})\right),
        \end{equation}
    where $\ell(S)$ is the index of the family of the labels of the colored component $S$. Since the injective traces in the product can also be written in terms of (non-injective) traces, asymptotic traffic independence entirely characterizes the limiting traffic distribution of $\mathbf{A}_{N}=\mathbf{A}_{N}^{(1)}\cup\cdots\cup\mathbf{A}_{N}^{(L)}$ in terms of the traffic distributions of the $\mathbf{A}_{N}^{(l)}$'s.
    \end{definition}

\subsection{Limiting moments of $\mu_{X_{N}}$}

    In \cite{male_traffic_2020}, it is shown that the families $\{V_1,V_1^*\},\ldots,\{V_d,V_d^{*}\},\{D_1,D_1^{*}\},\ldots,\{D_m,D_m^*\}$,\newline $\{\Delta_1,\Delta_1^*\},\ldots,\{\Delta_t,\Delta_t^{*}\},\{D_1'\},\ldots,\{D_{r_2}\}$ are asymptotically traffic independent.

    To avoid carrying an unnecessary factor for the rest of the proof, we study the matrix $X'_N:=(2d+2m+2t+r_2)X_{N}$. We compute the expectation of the limiting moments of $\mu_{N}$, its spectral measure, using traffic independence. 
    Let $n\in\N$ and $x$ be a formal variable. Let $T_{n}$ be the *-test graph consisting of a simple cycle of length $n$ where each edge is labeled $x$ (i.e. $\forall\,e\in E,\gamma(e)=x,\epsilon(e)=1$). We have
    \begin{equation}
        \E\left[\int t^{n}\dd\mu_{X'_N}(t)\right]=\E\left[\frac{1}{N}\tr((X'_N)^{n})\right]=\tau_{N}(T_{n}(X'_N)).
    \end{equation}
    Note that there is a one to one correspondence between matrices in the sum defining $X_N$ in Equation \eqref{eq: matrix model} and elements of $S$ the natural generating set of $G$. We often identify elements of $S$ with their matrix counterpart when no confusion is possible.
    We first develop the $n$-th power of $X'_N$ as follows.
    \begin{align}
        \E\left[\frac{1}{N}\tr((X'_N)^{n})\right]&=\sum_{\psi:[n]\rightarrow S}\E\left[\frac{1}{N}\tr(\psi(1)\cdots \psi(n))\right],\\
        &=\sum_{\psi:[n]\rightarrow S}\tau_{N}(T_{\psi}(\mathbf{V},\mathbf{D},\mathbf{\Delta},\mathbf{D'})),
    \end{align}
    where, $T_{\psi}$ is a cycle of length $n$ but whose edges are labeled by $\psi$ and bold capital letters stand for the family of the corresponding matrices. We often get rid of those in the following. We now use Equation \eqref{traffic freeness} and Examples \ref{diagonal} and \ref{uniform} to compute the limit of each $\tau_{N}(T_{\psi})$. For a colored component $C$, we denote $\ell(C)=1$ if it is labeled by some $v_i,v_i^{*}$ and $\ell(C)=2$ otherwise, then we have
    \begin{equation}\label{lamplighter}
        \lim_{N\rightarrow\infty}\tau_{N}(T_{\psi})=\sum_{\substack{\pi\in\mathcal{P}(V_{\psi}),\\\mathcal{GCC}(T_{\psi}^{\pi})\\ \text{ is a tree}}}\prod_{\substack{C\in\mathcal{CC}(T_{\psi}^{\pi}),\\\ell(S)=1}}\1_{C\text{ is a directed line}}\prod_{\substack{C\in\mathcal{CC}(T_{\psi}^{\pi}),\\\ell(S)=2}}\1_{|V_{S}|=1,E_{S} \text{ is valid}},
    \end{equation}
    where we define the validity of a test graph as follows.
    
    \begin{definition}\label{def: valid}
        A test-graph $T$ labeled by one of the families $\{d_j,d_j^*\}_{1\leq j \leq m},$ $\{\delta_j,\delta_j^*\}_{1\leq j \leq t},$ $\{d'_j\}_{1\leq j \leq r_2}$ is valid whenever either,$(i)$ it is labeled by some $d_j$ and $|E_1|-|E_*|=0$, either $(ii)$ it is labeled by some $\delta_j$ and $|E_1|-|E_*|=0\,\mathrm{mod}\,q_j$, or $(iii)$ it is labeled by some $d'_j$ and $|E|=0[2]$. 
    \end{definition}
    Note that $\psi$ encodes a path on the Cayley graph of $G$ with generators $S$. Let us now show that if $\psi$ encodes a self-returning path of length $n$, then there is a unique $\pi\in\mathcal{P}(V_{\psi})$ that satisfies all the conditions of Equation \eqref{lamplighter}, and if $\psi$ is not a self-returning path, then there is no partition satisfying all the conditions.
    \begin{proposition}
        We have 
        \begin{equation}
            \lim_{N\rightarrow\infty}\E\left[\int t^{n}\dd\mu_{X'_N}(t)\right]=|S|^{n}p_{n}(id,id)=|S|^{n}\int t^{n}\dd\mu(t).
        \end{equation}
    \end{proposition}
    \begin{proof}
        Let $n\in\N$ and $\psi:[n]\rightarrow S$ be a walk on the Cayley graph of the lamplighter group of length $n$ as described above. First, a connected component $\mathcal{S}$ with $\ell(\mathcal{S})=2$  may have a non-zero contribution at the limit only if it has only one vertex. Therefore,
        \begin{equation}
            \lim_{N\rightarrow\infty}\tau_{N}(T_{\psi}(V_{N},D_{N}))=\lim_{N\rightarrow\infty}\tau_{N}(T_{\psi}'(V_{N},D_{N})),
        \end{equation}
        where $T_{\psi}'$ is obtained from $T_{\psi}$ by identifying the source and target of all edges in $T_{\psi}$ labeled by $d,d'$ or $\delta$. 
        
        Now, $T'_{\psi}$ consists in a cycle of length $k\leq n$ labeled with $v_i$ for $1\leq i\leq d$ and decorated at each vertex with some (possibly none) loops labeled $d_j,d'_{j'}$ and $\delta_{l}$ for $1\leq j \leq m$, $1\leq j'\leq r_2$ and $1\leq l\leq t$,  such that there are $n-k$ such loops. We denote $\psi|^\Gamma$ the map $\psi$ co-restricted to the values $\hat{S_{\Gamma}}$, i.e. $v_i,v_i^*$. Combining Example \ref{uniform} and Equation \eqref{traffic freeness}, if $\psi|^\Gamma$ encodes a closed path on the Cayley graph of $\Z^{*d}$, there is a unique partition $\pi$ that (may or may not depending on the contribution of the diagonal matrices) contribute to the limit and otherwise the limit is zero. We thus have shown 
        \begin{equation}\label{balanced}
            \lim_{N\rightarrow\infty}\tau_{N}(T_{\psi}(V_{N},D_{N}))=\1_{\psi|^\Gamma\text{ encodes a closed path}}\lim_{N\rightarrow\infty}\tau^0_{N}((T_{\psi}')^{\pi}).
        \end{equation}
        Note that the condition $\psi|^\Gamma$ encodes a closed path is a required condition for a path on the Cayley graph of $G$ to be self-returning. Furthermore, for the limit on the right-hand side of \eqref{balanced} to not be zero, each connected component $\mathcal{S}$ of $\left(T_{\psi}'\right)^{\pi}$ with $\ell(\mathcal{S})=2$ must be valid in the sense given above, which is equivalent to the fact that at each point $\gamma$ of $\Gamma$, the sequence of $\lambda\in S_\Lambda$ such that the path uses the generator $\hat{\lambda}$ when at position $\gamma$ reduces to $e_\Lambda$, the neutral element of $\Lambda$. Hence, if $\psi$ is not a path from $id$ to $id$ on the Cayley graph of $G$, then either the induced walk on $\Gamma$ (by considering only edges with $\psi(e)\in \hat{S_\Gamma}$) is not self returning, either the induced walk on $\Gamma$ is self-returning (then there is a unique partition $\pi$ such that the injective trace may not be zero) but at least one vertex of the partitioned test-graph has a colored component which is not valid. In both cases the limit is zero. Now, if $\psi$ is self-returning, there exists a unique $\pi$ such that $(T_{\psi}')^{\pi}$ induces a closed walk on the Cayley graph of $\Gamma$ and every vertex in $(T_{\psi}')^{\pi}$ has connected components labeled by $d_j,d'j,\delta_j$ that are valid, then the limit is one using Equation \eqref{lamplighter} and Examples \ref{diagonal} and \ref{uniform}.
    \end{proof}

    \subsection{Convergence in probability}

    \begin{proof}[Proof of theorem \ref{measure convergence}]
        It is shown in \cite{male_traffic_2020} that a family of random uniform permutation matrices and a family of diagonal matrices that have a limit in traffic distribution both satisfy the following factorization property. For any finite family of *-graph polynomials $(g_{i})_{i\in I}$, for $\mathbf{A}_{N}=\mathbf{V}_{N},\mathbf{D}_{N},\mathbf{D'}_{N}$ or $\mathbf{\Delta}_N$, we have
        \begin{equation}\label{factorization property}
            \lim_{N\rightarrow\infty}\E\left[\prod_{i\in I}\frac{1}{N}\tr(g_{i}(A_{N}))\right]=\prod_{i\in I}\lim_{N\rightarrow\infty}\E\left[\frac{1}{N}\tr(g_{i}(A_{N}))\right].
        \end{equation}
        Moreover, since independent permutation matrices and diagonal matrices are asymptotically traffic free, the family $\{\mathbf{V}_{N},\mathbf{D}_{N},\mathbf{D'}_{N},\mathbf{\Delta}_N\}$ also satisfies this factorization property (see \cite[Theorem 1.8]{male_traffic_2020}). Hence, for any $k\in\N$, taking a family of two graph polynomials $g_{1}=g_{2}$ which consist in a cycle of length $k$ where all the edges are labeled $x$, one has
        \begin{equation}
            \lim_{N\rightarrow\infty}\E\left[\left(\frac{1}{N}\tr((V_{N}+V_{N}^{*}
            +D_{N})^{k})\right)^{2}\right]=\left(\lim_{N\rightarrow\infty}\E\left[\frac{1}{N}\tr((V_{N}+V_{N}^{*}
            +D_{N})^{k})\right]\right)^{2}.
        \end{equation}
        Using Bienaymé-Tchebychev's inequality, we thus have proven the convergence in probability of all the moments of $\mu_{X_{N}}$ towards the moments of $\mu$. Since both measure have a bounded support, this result is also true when integrating any bounded continuous function $f$. We have thus proven Theorem \ref{measure convergence}.
    \end{proof}

    \section{Central limit theorem for the moments of the lamplighter matrix}

    In this section, we only consider the case where $\Gamma = \Z$, so $d=1$. As in the previous section, in order to avoid carrying an unnecessary factor, we consider the matrix $X'_N:=(2+2m+2t+r_2)X_{N}$. Note that this only amounts to rescaling the covariance. The goal of this section is to provide a second order asymptotic for the moments of the matrix $X'_{N}$, namely Theorem \ref{TCL}. We denote $\mathbf{A}_{N}$ the family $\{V_{N},V_N^*,\mathbf{D}_{N},\mathbf{D'}_{N},\mathbf{\Delta}_N\}$. Let us recall the definition of the quantity at stake,
    \begin{equation}\label{first family}
	   Z_{N}(n):=\frac{1}{\sqrt{N}}\left[\mathrm{Tr}((X'_N)^{n})-\E(\mathrm{Tr}((X'_N)^{n}))\right].
    \end{equation}
    In order to show the convergence of the process $(Z_N(n))_{n\in\N}$ toward a Gaussian process, we rather show the convergence of the  following process $(Z_N(T))_{T}$ which implies the convergence of $(Z_N(n))_n$.
    
    Let $n\in\N$, $k_{1},\cdots,k_{n}\in \N$ and $\psi_{1},\cdots,\psi_{n}$ such that for all $1\leq i \leq n$, $\psi_{i}:[k_{i}]\rightarrow S$, where we recall that $S$ is the set of generators and their inverse of the group $G$ and is in one-to-one correspondence with the matrices appearing in the sum defining $X_N$. Let $T_{\psi_{i}}$ be a cycle of length $k_{i}$, whose edges are labeled according to $\psi_{i}$ (the first one is labeled $\psi(1)$, then following the direction of the edges, the next one is labeled $\psi(2)$ and so on). For each $i$ let $\pi_{i}$ be a partition of the vertices of $T_{\psi_{i}}$. Lastly, we denote $T_{i}:=T_{\psi_{i}}^{\pi_{i}}$ for simplicity. We define
	
	\begin{equation}
		(Z_{N}(T))_T:=\left(\frac{1}{\sqrt{N}}\left[\mathrm{Tr}^{0}(T_{i}(\mathbf{A}_{N})))-\E(\mathrm{Tr}^{0}(T_{i}(\mathbf{A}_{N})))\right]\right)_T,
	\end{equation}
	where the family is among all $T$ obtained from $n\in\N$, $k\in\N$, $\psi$ such that for all $\psi:[k]\rightarrow S$ and all $\pi\in\mathcal{P}(T_{\psi})$. Since each $Z_N(n)$ is a linear combination of $Z_N(T)$'s the convergence of the latter process toward a Gaussian process is stronger.
    
\subsection{Technical results}
     Prior to doing so, we state two technical lemmas giving the asymptotic for the injective trace or the product of injective trace of several test-graphs. We use the terminology \emph{underlying graph} of a test-graph $T$ as the graph obtained from $T$ after forgetting its labels, the direction and the multiplicity of its edges.
     
    \begin{lemma}\label{asymptotics}
        Let $C$ be a test-graph whose underlying graph is a cycle of length $c$ and $L$ be a test-graph whose underlying graph is a line of length $l$. Let $V_N$ be a uniform random permutation matrix. By a direct computation, one has
	       \begin{align*}
		      \E\left[\mathrm{Tr}^{0}(C(V_{N}))\right]&=N\P(\mathcal{U}_{N}\in\mathrm{Cyc}_{c})=1,\\
            \E\left[\mathrm{Tr}^{0}(L(V_{N}))\right]&=N\P(\mathcal{U}_{N}\in\mathrm{Cyc}_{\geq l})=N(1-\frac{l-1}{N}),
	       \end{align*}
        where $\mathcal{U}_{N}$ is a uniform random variable on $\{1,\cdots,N\}$ and the notations $\mathcal{U}_{N}\in\mathrm{Cyc}_{c}$ and $\mathcal{U}_{N}\in\mathrm{Cyc}_{\geq l}$ correspond to "$U_{N}$ belongs to a cycle of length $c$ of $V_{N}$" and "$U_{N}$ belongs to a cycle of length at least $l$ of $V_{N}$" respectively. 

        Moreover, if a test-graph $T$ has an underlying graph which is not a cycle nor a directed line, then, $\E\left[\mathrm{Tr}^{0}(T(V_{N}))\right]=0$.
    \end{lemma}

    \begin{lemma}
        Let $C_{1},\cdots,C_{K}$ be test-graphs whose underlying graphs are cycles of length $c_{1},\cdots,c_{K}$ respectively. Let $L_{1},\cdots,L_{M}$ be test-graphs whose underlying graphs are lines of length $l_{1},\cdots,l_{M}$ respectively. We have the following equality,
        \begin{multline}\label{mix line and cycles}
            \E\left[\prod_{k=1}^{K}\mathrm{Tr}^{0}(C_{k}(V_{N}))\prod_{m=1}^{M}\mathrm{Tr}^{0}(L_{m}(V_{N}))\right]=\\N^{K+M}\E\left[\prod_{k=1}^{K}\frac{c_{k}\eta_{c_{k}}}{N}\prod_{m=1}^{M}\left(1-\frac{\sum_{j=1}^{l_{m}-1}j\eta_{j}}{N}\right)\right]
        \end{multline}
        where $L:=\mathrm{max}\{c_{1},\cdots,c_{K},l_{1},\cdots,l_{M}\}$ and for all $1\leq j\leq L$, let $\eta_{j}$ is the number of cycles of length $j$ of $V_{N}$
    \end{lemma}
    
    \begin{proof}
    Let $V_{k},E_{k}$ be the set of vertices and edges respectively of $C_{k}$ and $V'_{m},E'_{m}$ be the set of vertices and edges respectively of $L_{m}$ for all $1\leq k\leq K$, $1\leq m \leq M$. Developing each injective trace, we obtain,
        \begin{multline}
            \E\left[\prod_{k=1}^{K}\mathrm{Tr}^{0}(C_{k}(V_{N}))\prod_{m=1}^{M}\mathrm{Tr}^{0}(L_{m}(V_{N}))\right]\\
            =\sum_{\substack{\phi_{1},\cdots\phi_{K}\\\phi_{k}:V_{k}\rightarrow[N]\\ \text{ injective}}}\sum_{\substack{\psi_{1},\cdots\psi_{M}\\\psi_{m}:V'_{m}\rightarrow[N]\\ \text{ injective}}}\E\left(\prod_{k=1}^{K}\prod_{e\in E_{k}}V_{N}^{\epsilon(e)}(\phi_{k}(e))\prod_{m=1}^{M}\prod_{e\in E'_{m}}V_{N}^{\epsilon(e)}(\psi_{m}(e))\right).
        \end{multline}
        Let us now arbitrarily choose a vertex $v_{k}$ in each $V_{k}$ (resp. $v'_{m}$ in each $V'_{m}$). Note that for all $1\leq k\leq K$ (resp. all $1\leq m \leq M$), once one chooses $\phi_{k}(v_{k})$ (resp. $\psi_{m}(v'_{m})$), since the matrix has only a single non zero entry per row and per column, and since the $\phi_{k}$ (resp. $\psi_{m}$) are injective, the rest of the values of $\phi_{k}$ (resp. $\psi_{m}$) are entirely determined. Hence, artificially adding a $N\frac{1}{N}$ factor for each of the $\phi_{k}$ and the $\psi_{m}$ and letting $\mathcal{U}_{1},\cdots,\mathcal{U}_{K}$, $\mathcal{U'}_{1},\cdots,\mathcal{U'}_{M}$ be i.i.d. random uniform variables on $\{1,\cdots,N\}$, we obtain 
        \begin{multline}
            \E\left[\prod_{k=1}^{K}\mathrm{Tr}^{0}(C_{k}(V_{N}))\prod_{m=1}^{M}\mathrm{Tr}^{0}(L_{m}(V_{N}))\right]\\
            =N^{K+M}\E\left(\prod_{k=1}^{K}\1_{\mathcal{U}_{k}\in\mathrm{Cyc}_{c_{k}}}\prod_{m=1}^{M}\1_{\mathcal{U'}_{m}\in\mathrm{Cyc}_{\geq l_{m}}}\right).
        \end{multline}
        One then has to condition on the $\eta_{j}$'s for $1\leq j \leq L$ to obtain the above result.
    \end{proof}

    \subsection{Proof of Theorem \ref{TCL}}
    
    By Wick's formula, a centered Gaussian process is characterized by its covariance. Therefore, we only study the asymptotics of 
    \begin{equation}\label{equation at stake}
        \E(Z_N(T_1)\cdots Z_N(T_n)),
    \end{equation}
    and show that it can be expressed in terms of the covariances as in Wick's formula.
	In the following, we assume that each edge labeled $d_j,d'_j$ or $\delta_j$ is a loop, otherwise the quantity in Equation \eqref{equation at stake} is trivially zero. Furthermore, the graph obtained from $T_i$ after deleting all the loops labeled $d_j,d'_j$ or $\delta_j$ forgetting the multiplicity of its edges is a test graph in the variables $v,v^*$ and is denoted $\bar T_i$. Hence, it must either be a line or a cycle since $V_N$ has a single non zero entry per row and per column. We only consider this cases in the following and denote for short "$\bar T_i$ is a line" or "$\bar T_i$ is a cycle" respectively. In both cases, we denote $|\bar T_i|$ the size of the underlying line (or cycle) i.e. its number of edges.
    
    We also say "$T_{i}$ is valid" whenever each of its colored component in one of the families of diagonal matrices is valid in the sense of Definition \ref{def: valid}. We denote $P:=\{i\in [n], \; T_{i} \text{ is valid}\}$ and $I:=[n]\setminus P$. Note that for all $i\in I$, $\E(\mathrm{Tr}^{0}(T_{i}(\mathbf{A}_{N})))=0$ by Example \ref{diagonal}. In the following, we omit the dependency in $\mathbf{A}_{N}$ in all notations since there is no risk of confusion. Developing the product, we get
	
	\begin{align}
		\E\left[Z_{N}(T_{1})\cdots Z_{N}(T_{n})\right]&=\E\left[\prod_{i\in P}Z_{N}(T_{i})\prod_{i\in I}\frac{1}{\sqrt{N}}\mathrm{Tr}^{0}(T_{i})\right]\\\label{developped}
		&=N^{-n/2}\sum_{B\subset P}(-1)^{|P|-|B|}\E\left[\prod_{i\in I\sqcup B}\mathrm{Tr}^{0}(T_{i})\right]\prod_{i\in P\setminus B}\E\left[\mathrm{Tr}^{0}(T_{i})\right].
	\end{align} 
	For the second term, using Lemma \ref{asymptotics} we have 
    \begin{align}\label{second part}
        \prod_{i\in P\setminus B}\E\left[\mathrm{Tr}^{0}(T_{i})\right]&=N^{|P|-|B|}\prod_{i\in P\setminus B}\left(\1_{T_i\text{  is a cycle}}\frac{1}{N}+\1_{ T_i\text{ is a line}}\left(1-\frac{|T_i|-1}{N}\right)\right)\\
        &=:N^{|P|-|B|}\prod_{i\in P\setminus B}h(T_i).
    \end{align}
    
    Let us now focus on the first term of Equation \eqref{developped}. We recall for all $1\leq i \leq n$, the map $\gamma_i:E_i\rightarrow S$ labels each edge by a certain generator and we recall that $S=\hat{S_\Gamma}\sqcup\hat{S_{\Lambda}}$, where $S_\Gamma$ is the set of generators of $\Gamma$ and their inverse and similarly for $S_\Lambda$ whereas the $\hat{}$ denotes the natural inclusion of each groups into $G=\Lambda\wr\Gamma$. Developing each injective trace, we obtain,

    \begin{align}
        \prod_{i\in I}\mathrm{Tr}^{0}(T_{i})&=\sum_{\substack{(\phi_{i})_{i\in I}\\ \phi_{i}:V_{i}\rightarrow [N] \\ \text{injective}}}\prod_{i\in I}\prod_{e\in E_{i}}A_{\gamma_{i}(e)}(\phi_{i}(e))\\
        &=\sum_{\sigma \in P(\sqcup_{i\in I}V_{i})}\sum_{(\phi_{i})_{i\in I}\in S_{\sigma}}\prod_{i\in I}\prod_{e\in E_{i}}A_{\gamma_{i}(e)}(\phi_{i}(e))\\\label{equation developped}
        &=\sum_{\pi\in P(I)}\sum_{\substack{\sigma \in P(\sqcup_{i\in I}V_{i}),\\ \overline{\sigma}=\pi}}\sum_{(\phi_{i})_{i\in I}\in S_{\sigma}}\prod_{i\in I}\prod_{e\in E_{i}}A_{\gamma_{i}(e)}(\phi_{i}(e)),
    \end{align}
    where notations are as below:
    \begin{itemize}
        \item $P(\sqcup_{i\in I}V_{i})$ is the set of partitions $\sigma$ of the disjoint union of the $V_{i}$'s for $i\in I$ such that each block of $\sigma$ contains at most one element of each $V_{i}$.
        \item For a given $\sigma \in P(\sqcup_{i\in I}V_{i})$, $\overline{\sigma}$ is the element of $P(I)$ defined by $i\sim_{\overline{\sigma}}j\Leftrightarrow\exists v\in V_{i},v'\in V_{j}$ such that $v\sim_{\sigma}v'$. In that way, $\overline{\sigma}$ tells us which of the $V_{i}$'s are connected through $\sigma$.
        \item $S_{\sigma}$ is the set of injective maps $(\phi_{i}:V_{i}\rightarrow[N])_{i\in I}$ such that for any $v\in V_{i},v'\in V_{j}$, we have $\phi_{i}(v)=\phi_{j}(v')\Leftrightarrow v\sim_{\sigma}v'$. Equivalently, $S_{\sigma}$ is the set of injective functions $\Phi:V_{I}^{\sigma}\rightarrow[N]$ (where $V_{I}^{\sigma}$ is defined below).
    \end{itemize}

    For the following, for a given $\sigma \in P(\sqcup_{i\in I}V_{i})$, we denote $T_{I}^{\sigma}=(V_{I}^{\sigma},E_{I}^{\sigma}):=\left(\bigsqcup_{i\in I}T_{i}\right)^{\sigma}$ and $\mathcal{CC}(T_{I}^{\sigma})$, the set of all its connected components. 

    Since for every $i\in I$, there exists a vertex with a colored component attached to it that is not valid, this vertex must be associated via $\sigma$ to another vertex of a $T_{j}$ for a certain $j\neq i$ otherwise the expectation of the injective trace is zero. Thus, only the partitions $\pi$ that have all blocks of size at least two may count. We denote $P_{\geq2}(I)$ the set of all such partitions. Note that for $\pi\in P_{\geq 2}(I)$, for all $\sigma\in P(\sqcup_{i\in I}V_{i})$ such that $\overline{\sigma}=\pi$, we have $\#\mathcal{CC}(T_{I}^{\sigma})\leq \frac{|I|}{2}$. 
    Using the independence of the entries of all families of matrices, we write
    \begin{multline}
        \E\left[\prod_{i\in I\sqcup B}\mathrm{Tr}^{0}(T_{i})\right]=\sum_{\pi\in P_{\geq 2}(I)}\sum_{\substack{\sigma \in P(\sqcup_{i\in I}V_{i}),\\ \overline{\sigma}=\pi}}\sum_{(\phi_{i})_{i\in I}\in S_{\sigma}}\sum_{\substack{(\phi_{i})_{i\in B}\\\phi_{i}\text{ injective}}}\\
        \E\left[\prod_{i\in I\sqcup B}\prod_{\substack{e\in E_{i}\\\gamma_i(e)\in \hat{S_\Gamma}}}V^{\varepsilon_i(e)}(\phi_{i}(e))\right]\underbrace{\E\left[\prod_{i\in B\sqcup I}\prod_{\substack{e\in E_{i}\\\gamma_i(e)\in\hat{S_\Lambda}}}A_{\gamma_i}(e)(\phi_{i}(e))\right]}_{=:f(B,I,(\phi_{i})_{i\in I\sqcup B})}.
    \end{multline}
    Note that since $B\subset P$, for all $i\in B$,the quantity $f(B,I,(\phi_{i})_{i\in I\sqcup B})$ does not depend on $B$. Furthermore since the entries of each diagonal matrices are i.i.d. $f(B,I,(\phi_{i})_{i\in I\sqcup B})$ only depends on $\sigma$ (where $(\phi_{i})_{i\in I}\in S_{\sigma}$). We denote $f(I,\sigma):=f(B,I,(\phi_{i})_{i\in I\sqcup B})$ for all $(\phi_{i})_{i\in I}\in S_{\sigma}$, all $B\subset P$ and all $(\phi_{i})_{i\in B}$ injective. Therefore we can write the above term as
    \begin{multline}\label{séparer v et d}
        \E\left[\prod_{i\in I\sqcup B}\mathrm{Tr}^{0}(T_{i})\right]=\sum_{\pi\in P_{\geq 2}(I)}\sum_{\substack{\sigma \in P(\sqcup_{i\in I}V_{i}),\\ \overline{\sigma}=\pi}}f(I,\sigma)\times\\
        \E\left[\prod_{j\in B}\mathrm{Tr}^{0}(\bar T_{j})\sum_{(\phi_{i})_{i\in I}\in S_{\sigma}}\prod_{i\in I}\prod_{\substack{e\in E_{i}\\\gamma_i(e)\in \hat{S_\Gamma}}}A_{\gamma_i(e)}(\phi_{i}(e))\right],
    \end{multline}
     where for a test-graph $T$, we recall that the test-graph $\bar T$ is the test graph obtained after forgetting the edges labeled $d_j,d'_j,\delta_j$, and forgetting the multiplicity of the remaining edges.

    The following proposition simplifies the calculation of the above expectation.

    \begin{proposition}
    Let $I,P$ be defined as above and $B\subset P$. We denote $\mathfrak{C}\mathfrak{C}(T)$ the set of connected components of $T$, not to be confused with $\CC\CC(T)$ the set of colored component with respect to some families. Let $\sigma\in P(\sqcup_{i\in I}V_{i})$, then 
        \begin{multline}\label{equation proposition}
            \E\left[\prod_{j\in B}\mathrm{Tr}^{0}(\bar T_{j})\sum_{(\phi_{i})_{i\in I}\in S_{\sigma}}\prod_{i\in I}\prod_{\substack{e\in E_{i}\\\gamma_i(e)\in\hat{S_\Gamma}}}A_{\gamma_i(e)}(\phi_{i}(e))\right]\\
            =\E\left[\prod_{j\in B}\mathrm{Tr}^{0}(\bar T_{j})\prod_{T \in \mathfrak{CC}(T_{I}^{\sigma})}\mathrm{Tr^{0}}(\bar T)\right]+O\left(N^{|B|+|\mathfrak{CC}(T_{I}^{\sigma})|-1}\right).
        \end{multline}
    \end{proposition}
    \begin{proof}
    Let $\sigma\in P(\sqcup_{i\in I}V_{i})$. Throughout the proof, in order to alleviate notations, we assume that for all $i\in B$, $T_{i}=\bar T_i$ and similarly with all $T\in\mathfrak{CC}(T_{I}^{\sigma})$, we assume $\bar T=T$.
    
    Starting from the right-hand side of Equation \eqref{equation proposition}, developing the injective trace, we have
    \begin{multline}\label{injective trace connected components}
        \prod_{T\in \mathfrak{CC}(T_{I}^{\sigma})}\mathrm{Tr^{0}}(T)=\sum_{(\phi_{i})_{i\in I}\in S_{\sigma}}\prod_{i\in I}\prod_{e\in E_{i}}A_{\gamma_i(e)}(\phi(e))\\
        +\sum_{\substack{\sigma'>\sigma\\\text{and }\overline{\sigma'}>\overline{\sigma}}}\sum_{(\phi_{i})_{i\in I}\in S_{\sigma'}}\prod_{i\in I}\prod_{e\in E_{i}}A_{\gamma_i(e)}(\phi(e)).
    \end{multline}
    Note that for any $\sigma'$ such that $\overline{\sigma'}>\overline{\sigma}$, we have $|\mathfrak{CC}(T_{I}^{\sigma'})|\leq|\mathfrak{CC}(T_{I}^{\sigma})|-1$. Hence, for any $\sigma'>\sigma$ verifying $\overline{\sigma'}>\overline{\sigma}$,
    \begin{align*}
        &N^{-|B|-|\mathfrak{CC}(T_{I}^{\sigma})|+1}\Bigg|\sum_{(\phi_{i})_{i\in I}\in S_{\sigma'}}\E\left[\prod_{j\in B}\mathrm{Tr}^{0}(T_{j})\prod_{i\in I}\prod_{e\in E_{i}}A_{\gamma_i(e)}(\phi(e))\right]\Bigg|\\
        &\leq N^{-|B|-|\mathfrak{CC}(T_{I}^{\sigma'})|}\sum_{(\phi_{i})_{i\in I}\in S_{\sigma'}}\E\left[\prod_{j\in B}\mathrm{Tr}^{0}(T_{j})\prod_{i\in I}\prod_{e\in E_{i}}A_{\gamma_i(e)}(\phi(e))\right]\\
        &\leq N^{-|B|-|\mathfrak{CC}(T_{I}^{\sigma'})|}\E\left[\prod_{j\in B}\mathrm{Tr}^{0}(T_{j})\prod_{T\in\mathfrak{CC}(T_{I}^{\sigma'})}\mathrm{Tr}^{0}(T)\right]=O(1).
    \end{align*}
    Since the number of such partitions $\sigma'$ is independent from $N$, this concludes the proof of the proposition.
    \end{proof}

    We now use Equation \eqref{mix line and cycles} denoting $B:=B_{c}\sqcup B_{l}$ where $i\in B_{c}\Leftrightarrow T_{i}$ is a cycle and $i\in B_{l}\Leftrightarrow T_{i}$ is a line and similarly $\mathcal{CC}(T_{I}^{\sigma}):=\mathcal{CC}(T_{I}^{\sigma})_{c}\sqcup \mathcal{CC}(T_{I}^{\sigma})_{l}$. Recall that we denote $|T_{i}|$ the size of $T_{i}$ (either as a line or as a cycle) and $|S|$ the size of $S$ either as a line or as a cycle. Recall also that $\eta_j$ is the number of cycles of length $j$ in the permutation $V_N$. We denote, for $i\in B$, $k(i)=\frac{|T_i|\eta_{|T_{i}|}}{N}$ if $i\in B_{c}$ and $k(i)=1-\frac{\sum_{j=1}^{|T_{i}|-1}j\eta_{j}}{N}$ if $i\in B_{l}$. Similarly, for $S\in\mathcal{CC}(T_{I}^{\sigma})$, we denote $k(S)=\frac{|S|\eta_{|S|}}{N}$ if $S\in \mathcal{CC}(T_{I}^{\sigma})_{c}$ and $k(S)=1-\frac{\sum_{j=1}^{|S|-1}j\eta_{j}}{N}$ if $i\in \mathcal{CC}(T_{I}^{\sigma})_{l}$. Combining Equations \eqref{mix line and cycles}, \eqref{séparer v et d} and \eqref{equation proposition}, we have

    \begin{multline}\label{first part}
         N^{-|B|-|\mathcal{CC}(T_{I}^{\sigma})|}\E\left[\prod_{i\in I\sqcup B}\mathrm{Tr}^{0}(T_{i})\right]\\
         =\sum_{\pi\in P_{\geq 2}(I)}\sum_{\substack{\sigma \in P(\sqcup_{i\in I}V_{i}),\\ \overline{\sigma}=\pi}}f(I,\sigma)\E\left[\prod_{i\in B}k(i)\prod_{S\in \mathcal{CC}(T_{I}^{\sigma})}k(S)\right]+O(N^{-1}).
    \end{multline}

    Recall that $n=|P|+|I|$, going back to Equation \eqref{developped}, using Equations \eqref{second part} and \eqref{first part}, we now have
    \begin{multline}
        \E\left[Z_{N}(T_{1})\cdots Z_{N}(T_{n})\right]=\sum_{\pi\in P_{\geq 2}(I)}\sum_{\substack{\sigma \in P(\sqcup_{i\in I}V_{i}),\\ \overline{\sigma}=\pi}}N^{-n/2+|P|+|\mathcal{CC}(T_{I}^{\sigma})|}\sum_{B\subset P}(-1)^{|P|-|B|}\\
        \prod_{i\in P\setminus B}h(i)\E\Bigg[\prod_{i\in B}k(i)\prod_{S\in \mathcal{CC}(T_{I}^{\sigma})}k(S)\Bigg]f(I,\sigma)+O\left(N^{-1-n/2+|P|+|\mathcal{CC}(T_{I}^{\sigma})|}\right).
    \end{multline}
    Factorizing the sum over $B\subset P$, we get
    \begin{multline}
        \E\left[Z_{N}(T_{1})\cdots Z_{N}(T_{n})\right]=\sum_{\pi\in P_{\geq 2}(I)}\sum_{\substack{\sigma \in P(\sqcup_{i\in I}V_{i}),\\ \overline{\sigma}=\pi}}N^{|P|/2+|\mathcal{CC}(T_{I}^{\sigma})|-|I|/2}\\
        \E\Bigg[\prod_{i\in P}\left(k(i)-h(i)\right)\prod_{S\in \mathcal{CC}(T_{I}^{\sigma})}k(S)\Bigg]f(I,\sigma)(1+O(N^{-1})).
    \end{multline}
    We now introduce a $N$ factor for each $i\in P$ as follows
    \begin{multline}
        =\sum_{\pi\in P_{\geq 2}(I)}\sum_{\substack{\sigma \in P(\sqcup_{i\in I}V_{i}),\\ \overline{\sigma}=\pi}}N^{-|P|/2+|\mathcal{CC}(T_{I}^{\sigma})|-|I|/2}\\
        \E\Bigg[\prod_{i\in P}N\left(k(i)-h(i)\right)\prod_{S\in \mathcal{CC}(T_{I}^{\sigma})}k(S)\Bigg]f(I,\sigma)(1+O(N^{-1})).
    \end{multline}
    Now, note that for all $i\in P$, the quantity $N\left(k(i)-h(i)\right)$ is either $|T_i|\eta_{|T_{i}|}-1$, either $|T_{i}|-1-\sum_{j=1}^{|T_{i}|-1}j\eta_{j}=-\sum_{j=1}^{|T_{i}|-1}(j\eta_{j}-1)$, depending on whether $T_{i}$ is a cycle or a directed line respectively. It is shown in \cite{kammoun_small_2023} that the family $(\eta_{j})_{1\leq j\leq \max{|T_{i}|}}$ converges in distribution as $N$ goes to infinity to a family of independent random variables $(\xi_{j})_{1\leq j\leq \max{|T_{i}|}}$ where $\xi_{j}$ is a random variable distributed like a Poisson of parameter $\frac{1}{j}$. Hence, since each $k(S)$ is smaller or equal to one, the expectation is bounded independently of $N$. Furthermore, $-|P|/2\leq0$ and $|\mathcal{CC}(T_{I}^{\sigma})|-|I|/2\leq 0$, so the terms in the above sum have a non-zero contribution at the limit only if $|P|=0$ and $|\mathcal{CC}(T_{I}^{\sigma})|=|I|/2=n/2$, which is in turn equivalent to $|P|=0$ and $\pi$ is a pairing of $[n]$ (denoted $\pi\in P_{2}(n)$), which means that all its blocks are of size 2. Furthermore, note that $f(I,\sigma)$ is 1 if $T_{I}^{\sigma}$ is valid and 0 otherwise.
    
    Summarizing all of this, denoting $T:=\bigsqcup_{i\in[n]}T_{i}$, we have
    \begin{multline}
        \E\left[Z_{N}(T_{1})\cdots Z_{N}(T_{n})\right]\\
        =\prod_{i\in[n]}\1_{T_{i}\text{ not valid}}\sum_{\pi\in P_{2}(n)}\sum_{\substack{\sigma \in P(\sqcup_{i\in [n]}V_{i}),\\ \overline{\sigma}=\pi}}\1_{T^{\sigma}\text{ valid}}\E \left[\prod_{S\in \mathcal{CC}(T^{\sigma})}k(S)\right](1+O(N^{-1})).
    \end{multline}
    Now, if the restriction to edges labeled $v,v^{*}$ of a connected component $S$ of $T^{\sigma}$ is a cycle, it brings a $N^{-1}$ factor, thus a vanishing contribution. Hence, $\E \left[\prod_{S\in \mathcal{CC}(T^{\sigma})}k(S)\right]=\prod_{S\in \mathcal{CC}(T^{\sigma})}\1_{S|_{v}\text{ directed line}}+o(1)$. The case $n=2$ gives
    \begin{equation}
        \E(Z_{N}(T_{1})Z_{N}(T_{2}))=\underbrace{\sum_{\sigma\in P(V_{1}\sqcup V_{2})}\1_{T_{1}\text{ not valid}}\1_{T_{2}\text{ not valid}}\tau^{0}\left((T_{1}\sqcup T_{2})^{\sigma}\right)}_{=:M^{(2)}(T_{1},T_{2})}+o(1).
    \end{equation}
    
    The general case for $n\geq3$ gives the Wick formula

    \begin{multline}
        \E\left[Z_{N}(T_{1})\cdots Z_{N}(T_{n})\right]=\sum_{\pi\in P_{2}(n)}\prod_{\{i,j\}\in\pi}\1_{T_{i}\text{ and } T_{j}\text{ not valid}}\sum_{\sigma \in P(V_{i}\sqcup V_{j})}\tau^{0}\left((T_{i}\sqcup T_{j})^{\sigma}\right)+o(1)\\
        =\sum_{\pi\in P_{2}(n)}\prod_{\{i,j\}\in\pi}M^{(2)}(T_{i},T_{j})+o(1),
    \end{multline}
    which characterizes the Gaussian distribution. This concludes the proof of Theorem \ref{TCL}.

    \section{Operator-valued free probability space}

    The goal of this section is to provide a natural framework to express the limiting objects introduced before. We start by recalling basics about operator-valued probability spaces, then we recall the notions of freeness with amalgamation over the diagonal before introducing the limiting space.

   In \cite{au_freeness_2021}, Au et al. showed that permutation invariant random matrices are asymptotically free over the diagonal. We use the fact that the families $\{V_{N},V_{N}^{*}\},$ $ \{D'_j\}_{1\leq j \leq r_2},$ $\{D_j,D_j^*\}_{1\leq j\leq t},$ $\{\Delta_j,\Delta_j^*\}_{1\leq j \leq m}$  are asymptotically free over the diagonal in order to compute the limit of the operator-valued free-cumulants of the family $\{V_{N},V_{N}^{*}\}\cup\{D'_j\}_{1\leq j \leq r_2}\cup\{D_j,D_j^*\}_{1\leq j\leq t}\cup\{\Delta_j,\Delta_j^*\}_{1\leq j \leq m}$ in Proposition \ref{prop: asymptotic op val cumulants}.
    
    \subsection{Basic framework of operator valued probability}
    We introduce some notions of operator-valued probability, the reader may look at \cite{mingo_free_2017} for a more detailed introduction to operator-valued non-commutative probability.

    \begin{definition}
        Let $\A$ be a unital algebra and $\B$ a unital subalgebra of $\A$. An \emph{operator-valued probability space} $(\A,\B,\Delta)$ consists of $\B\subset\A$ and a map $\Delta :\A\rightarrow\B$ satisfying
        \begin{align}
            \Delta(b)&=b\quad\quad\forall b\in\B,\\\label{conditionnal expectation}
            \Delta(b_{1}ab_{2})&=b_{1}\Delta(a)b_{2}\quad\quad\forall a\in\A,\quad\forall b_{1},b_{2}\in\B.
        \end{align}
        Such a map $\Delta$ is called a \emph{conditional expectation}.

        Let $K$ be an arbitrary index set. We write $\B\langle X_{k}: k \in K\rangle$ for the free non commutating algebra generated by $\B$ and the non commutating indeterminate $(X_{k})_{k\in K}$. We call an element of $\B\langle X_{k}: k \in K\rangle$ a $\B$-valued polynomial. The \emph{operator-valued distribution} (or $\Delta$-distribution for short) of a family $\mathbf{A} = (A^{(k)} )_{k\in K} \subset \A$ is the linear map of operator-valued moments
        \begin{equation}
            \Delta_{\mathbf{A}}:\B\langle X_{k}: k \in K\rangle\rightarrow\B,\quad P\mapsto\Delta(P(\mathbf{A})).
        \end{equation}

        We say that the families $\mathbf{A}_{1} = (A^{(k)}_{1} )_{k\in K} ,\cdots , \mathbf{A}_{L} = (A^{(k)}_{L} )_{k\in K} \subset \A$ are \emph{free with amalgamation over $\B$} (or free over $\B$ for short) if
        \begin{equation}\label{free over B}
            \Delta\left((P_{1}(\mathbf{A}_{l_{1}})-\Delta(P_{1}(\mathbf{A}_{l_{1}})))\cdots(P_{n}(\mathbf{A}_{l_{n}})-\Delta(P_{n}(\mathbf{A}_{l_{n}})))\right)=0,
        \end{equation}
        for any $\B$-polynomials $P_{1},\cdots,P_{n}\in\B\langle X_{k}: k \in K\rangle$ whenever $l_{j}\neq l_{j+1}$ for $1\leq j\leq n-1$.
    \end{definition}
    The ordinary freeness is contained in this definition when taking $\B=\C$.

    As in free probability, there is a notion of (operator-valued) free cumulants that is very similar to that of free probability. One only has to take care of the order of the variables in all the expressions.

    \begin{definition}
        In an operator-valued free probability space $(\A,\B,\Delta)$, we define the corresponding (operator-valued) free cumulants $(\kappa_{n}^{\B})_{n\in\N},\kappa_{n}^{\B}:\A^{n}\rightarrow\B$ by the moment cumulant formula :
        \begin{equation}\label{moment cumulant}
            \Delta(a_{1}\cdots a_{n})=\sum_{\pi\in NC(n)}\kappa_{\pi}^{\B}(a_{1},\cdots,a_{n}),
        \end{equation}
        where arguments of $\kappa_{\pi}^{\B}$ are distributed according to the blocks of $\pi$, but the cumulants are nested inside each other according to the nesting of the blocks of $\pi$.
    \end{definition}
    \begin{remark}
        Actually, using Equation \eqref{conditionnal expectation} one sees that 
    \begin{equation}
        \kappa_{n}^{\B}(b_{0}a_{1},b_{1}a_{2},\cdots,b_{n-1}a_{n}b_{n})=b_{0}\kappa_{n}^{\B}(a_{1}b_{1},\cdots,a_{n-1}b_{n-1},a_{n})b_{n},
    \end{equation}
    for all $a_{1},\cdots,a_{n}\in\A$ and all $b_{0},\cdots,b_{n+1}\in\B$. This can also be stated by saying that the cumulant $\kappa_{n}^{\B}$ is a map from the $\B$-module tensor product $\A^{\otimes_{\B}n}=\A\otimes_{\B}\A\otimes_{\B}\cdots\otimes_{\B}\A$.
    \end{remark}
    
    As in the free probability setting, we have a Möbius inversion formula for the operator-valued free cumulants, namely
    \begin{equation}\label{Mobius inversion}
            \kappa_{n}^{\B}(a_{1},\cdots,a_{n})=\sum_{\pi\in NC(n)}\mu(\pi,1_{n})\Delta_{\pi}(a_{1},\cdots,a_{n}),
    \end{equation}
     where $\Delta_{\pi}$ is defined from $\Delta$ as $\kappa_{\pi}$ is from $\kappa$ (that is, its arguments are nested according to the block of $\pi$), $\mu$ is the Möbius function on non-crossing partitions (see \cite[Chapter 10]{nica_lectures_2006}) and $1_{n}$ is the partition $\{\{1,2,\cdots,n\}\}$. A proof of this is given in \cite[Chapter 9]{mingo_free_2017} and is more of a combinatorial fact on multi-linear functionals on posets.
     
    As in the standard and in the free case, the freeness over $\B$ can also be stated in a simpler manner using the operator-valued free cumulants. Indeed, two elements $a_{1},a_{2}\in\A$ are free over $\B$ if and only if the unital $\B$-subalgebras of $\A$ generated by $a_{1}$ and $a_{2}$ (say $\A_{1}$ and $\A_{2}$ respectively), are free in the sense of Equation \eqref{free over B}. Moreover, this is also equivalent to : for all $c_{1},\cdots,c_{n}\in\A$ such that $c_{i}\in \A_{l_{i}}$ where $l_{i}$ is 1 or 2, then 
    \begin{equation}
        \kappa_{n}^{\B}(c_{1},\cdots,c_{n})=0,
    \end{equation}
    as soon as there exists $i,j\in[n]$ such that $l_{i}\neq l_{j}$. Two random variables are free over $\B$ if and only if the mixed operator-valued free cumulants vanish.

    \subsection{Cactus and asymptotic freeness over the diagonal}
    
    Let $\M_{N}(\C)$ be the algebra of complex $N\times N$ matrices. We consider in the following the diagonal map $\Delta:\M_{N}(\C)\rightarrow\mathcal{D}_{N}(\C)$ onto the subalgebra $\mathcal{D}_{N}(\C)$ of diagonal matrices defined by
    \begin{equation}
        \Delta((a_{i,j})_{i,j\in[N]})=(\delta_{i,j}a_{i,j})_{i,j\in[N]}.
    \end{equation}
    One can verify that $(\M_{N}(\C),\mathcal{D}_{N}(\C),\Delta)$ is an operator-valued probability space. We also have that $(\M_{N}(L^{\infty-}),\mathcal{D}_{N}(L^{\infty-}),\Delta)$ is an operator-valued probability space, where $\M_{N}(L^{\infty-})$ is the set of $N\times N$ random matrices whose entries have finite moments of all orders.

    One cannot directly define a notion of asymptotic freeness over $\mathcal{D}_{N}(L^{\infty-})$ (or even a convergence in operator-valued distribution) as the dimension of $\mathcal{D}_{N}(L^{\infty-})$ goes to infinity as $N$ grows. We thus consider a sub-operator-valued probability space of this one that allow us to see the elements of the subalgebra $\B_{N}$ as independent from $N$. The price for this change of setup is that we have to construct the algebra from the matrices we are interested in.

    \begin{definition}
        Let $\mathbf{A}_{N}=\{V_{N},V_{N}^{*},\mathbf{D}_{N}\}$ be a family of a uniform random permutation matrix $V_{N}$, its transpose and an independent family of independent diagonal matrices $\mathbf{D}_{N}$. We define $\A_{N}\subset \M_{N}(L^{\infty-})$ the smallest unital subalgebra containing $\mathbf{A}_{N}$ that is closed under the diagonal map $\Delta$. We denote by $\B_{N}$ its image under $\Delta$ : $\B_{N}=\Delta(\A_{N})\subset\A_{N}$.
    \end{definition}
    In the following, we take $\mathbf{D}_N$ to be the family of diagonal matrices corresponding to the group $\Lambda \wr \Gamma$, i.e. those of Equation \eqref{eq: matrix model}.
    \begin{definition}
        We say that a finite connected directed graph $G$ is a \emph{cactus} if every edge belongs to a unique simple cycle. In the case where each such cycle is directed (each edge have the same orientation), then we further specify that $G$ is an \emph{oriented cactus}. We write $\mathcal{C}$ for the set of all \emph{planted cactus-type monomials}, that is, graph monomials with $v_{in}=v_{out}$ whose underlying graph is an oriented cactus. We write $\mathcal{C}_{N}$ for the vector space generated by $(g(\mathbf{A}_{N}))_{g\in\mathcal{C}}$ . More generally, we define $\F$ to be the set of all graph monomials obtained by starting with a directed path
        \begin{equation*}
            \underset{v_{out}}{\cdot}\leftarrow\cdot\leftarrow\cdot  \cdots \cdot\leftarrow \underset{v_{in}}{\cdot}
        \end{equation*}
        and attaching oriented cacti to the vertices of this path
        \begin{equation*}
            \overset{\vee}{\underset{v_{out}}{\cdot}}\leftarrow\overset{\vee}{\cdot}\leftarrow\overset{\vee}{\cdot}  \cdots \overset{\vee}{\cdot}\leftarrow \overset{\vee}{\underset{v_{in}}{\cdot}}.
        \end{equation*}
        Similarly, we write $\F_{N}$ for the vector space generated by $(g(\mathbf{A}_{N}))_{g\in\F}$. We say that an element of $\F$ is a cactus field-type monomial.
    \end{definition}
    Note that $\Delta(\F)=\mathcal{C}$, where for any graph monomial $g$, $\Delta(g)$ denotes the graph monomial obtained by identifying $v_{in}$ and $v_{out}$ of $g$. This notation comes from the fact that for any $g\in\F$, $\Delta(g(\mathbf{A}_{N}))=\Delta(g)(\mathbf{A}_{N})$.
    
    The following proposition is proven in \cite{au_freeness_2021} and justifies the introduction of cacti and cactus fields.
    \begin{proposition}
        With the notations of the two above definitions, we have $\F_{N}=\A_{N}$ and $\mathcal{C}_{N}=\B_{N}$.
    \end{proposition}

    Thanks to this proposition, in this setup, one can define a notion of asymptotic freeness over $\B_{N}$ because now, the coefficients of the $\B_{N}$-valued polynomials are independent from $N$ : the graph structure does not change but we evaluate those graphs on larger and larger matrices. In \cite{au_freeness_2021}, it is shown that under some assumptions, two families of permutation invariant random matrices are asymptotically free over the diagonal.

\subsection{The limiting $C^*$ operator-valued probability space}

Throughout this subsection, we consider $\A:=\C\langle\F\rangle$, $\B:=\C\langle\mathcal{C}\rangle$ and $\Delta:\A\rightarrow\B$ defined as an operation on graphs extended by linearity (where $\mathcal{F,C}$ were defined in section 4)
Let us define
$$ L:=\bigoplus_{k\in\mathbb Z}\Lambda,\quad\text{so that }\quad  G:=L\rtimes\mathbb Z,$$ 
where the semi-direct product is obtained from the shift. We always identify $L$ with $L\times\{0\}\leq G$ so that $L$ is viewed as a subgroup of $G$. Let $\lambda_G$ be the left regular representation on $\ell^2(G)$.  We write
$$ \A_0:=C_r^*(G)=\overline{\lambda_G(\mathbb C[G])}^{\|\cdot\|_{B(\ell^2(G))}},\qquad \tau_G(x):=\langle x\delta_{e_G},\delta_{e_G}\rangle .$$
The functional $\tau_G$ is the canonical faithful tracial state on $C_r^*(G)$; equivalently, it extracts the coefficient of $e_G$ on $\mathbb C[G]$ (see \cite[Chapter 3]{nica_lectures_2006} or \cite[Section~2.5]{brown_ozawa_2008}).

For a finite sum in $\mathbb C[G]$, define
\begin{equation}\label{eq:compact-expectation-fourier}
    E_L\left(\sum_{g\in G}c_g\lambda_G(g)\right):=\sum_{\ell\in L}c_\ell\lambda_G(\ell).
\end{equation}
This map extends to a conditional expectation $E_L:C_r^*(G)\longrightarrow C_r^*(L).$ Indeed, if $V:\ell^2(L)\hookrightarrow\ell^2(G)$ is the natural isometry, then $E_L(x)=V^*xV$ on $C_r^*(G)$. It is easy to see for elements $x\in \lambda_G(\C[G])$ and one obtains the equality in $C^*_r(G)$ by taking the limit in operator norm.  Hence $E_L$ is unital and completely positive, since compression by an isometry has these properties \cite[Chapter~2]{paulsen_cb_2002}.

Recall that a $C^*$-operator-valued probability space consists of a unital $C^*$-algebra $\A_0$, a unital $C^*$-subalgebra $\B_0\subset\A_0$, and a conditional expectation $E:\A_0\to\B_0$; see \cite{mingo_free_2017}.  In the present setting, take $ \B_0:=C_r^*(L)\subset \A_0=C_r^*(G).$ The inclusion is the canonical one induced by the subgroup $L\leq G$; it is isometric for reduced group $C^*$-algebras \cite[Section~2.5]{brown_ozawa_2008}.
Formula \eqref{eq:compact-expectation-fourier} shows that $E_L$ range is $C_r^*(L)$ and
that it fixes this subalgebra pointwise.  Thus it is a norm-one projection
onto $C_r^*(L)$, and \cite{tomiyama_1957} implies the bimodule identity
$$ E_L(b_1xb_2)=b_1E_L(x)b_2,\qquad\forall x\in C_r^*(G),\quad\forall  b_1,b_2\in C_r^*(L).$$
Consequently,
\begin{equation*}
 (\A_0,\B_0,E_L)=\left(C_r^*(G),C_r^*(L),E_L\right)
\end{equation*}
is a $C^*$-operator-valued probability space.

Moreover, the expectation is compatible with the graph diagonal in the following sense.  Recall that $f\in\mathcal F$ is \emph{balanced} when every planted cycle contains equally many labels $v$ and $v^*$. For $f\in\F$, let $w(f)\in G$ be the group element obtained by reading the labels along the spine and along the fixed contours of the planted cacti. Define for $f \in \mathcal F$,
\begin{equation}\label{eq:Theta}
    \Theta(f):=
    \begin{cases}
        w(f),&f\text{ balanced},\\
        0,&f\text{ non-balanced}.
    \end{cases}
\end{equation}
It extends to a surjective unital $*$-homomorphism $\Theta:\A\to\mathbb C[G]$. Indeed, concatenating two cactus fields concatenates their contour words and preserves all planted cycles. Hence the fact that $f$ is balanced is multiplicative and $w(f_1f_2)=w(f_1)w(f_2)$ whenever both factors are balanced; if one factor is not balanced, its unbalanced cycle remains present.  Reversing all edges replaces the contour word by its inverse and preserves being balanced or not.  Thus $\Theta$ is a unital $*$-homomorphism.  

Furthermore, every word in the generators of $G$ is represented by a directed path (hence balanced), proving surjectivity.

Moreover, we have the identities
\begin{equation}\label{eq:compact-Theta-B}
    \Theta(\B)=\mathbb C[L], \qquad\Theta\circ\Delta=E_L\circ\Theta.
\end{equation}
Closing the spine of an admissible cactus field creates one additional cycle. This cycle is balanced exactly when the $\mathbb Z$-coordinate of $w(f)$ is zero, that is, when $w(f)\in L$.  This proves the second identity in \eqref{eq:compact-Theta-B}; the first follows by applying the same observation to closed cactus fields and by representing each element of $L$ by a balanced closed word.

In the following, we often omit $\lambda_G$ and therefore identify formal linear combination of elements of $G$ with their corresponding operator on $\ell^2(G)$. Thus, for $a\in\A$, $\Theta(a)$ may be viewed as an element of $\A_0$.

Let us consider the linear form $\phi$ on $\A$ defined on cactus-field type monomials by 
\begin{equation*}
    \phi(a)= \lim_{N\rightarrow\infty} \E\left(\frac{1}{N}\tr a(V_{N},D_{N})\right),
\end{equation*}
and extended by linearity. The previous section says precisely that a non-balanced cactus field contributes zero and that a balanced one contributes one if and only if $w(f)=e_G$. Consequently,
\begin{equation}\label{eq: state identification}
    \phi(a)=\tau_G(\Theta(a)),\qquad \forall a\in\A.
\end{equation}

    \section{Limiting $R$-transform}

With the results of the previous section, a natural question is to establish a formula for the limiting $R$-transform. We establish the limiting cumulants in Proposition \ref{prop: asymptotic op val cumulants} and derive the $R$-transform in Proposition \ref{analytic R transform}.

\subsection{Asymptotic operator-valued cumulants}

We keep the notations introduced in the last section. If $b\in\mathcal{C}$ is a cactus-type monomial, we denote its evaluation by $b_N=b(V_N,\mathbf D_N)\in\mathcal{A}_N$. We will use the following description of $\Theta$ on cacti explained in the previous subsection: if $b$ is balanced, then there exists a finitely supported function $f:\Z\to\Lambda$ such that $\Theta(b)=(f,0)$, whereas $\Theta(b)=0$ if $b$ is not balanced. With the identification $L=L\times\{0\}$, we have $\Theta(b)=\lambda_L(f)\in\B_0$ if $b$ is balanced.

\begin{proposition}[Asymptotic operator-valued cumulants]\label{prop: asymptotic op val cumulants}
    Let $r\geq1$, let $b_1,\ldots,b_n\in\mathcal C$ be cactus-type monomials, and let $\epsilon:[n]\to\{1,-1\}$. Then
    \begin{equation}\label{eq:target-limit}
        \lim_{N\to\infty}\E\left[\frac{1}{N} \mathrm{Tr}\left(\kappa_n^{\B_N}\bigl(V_N^{\epsilon(1)}b_{1,N},\ldots,V_N^{\epsilon(n)}b_{n,N}\bigr)\right)\right]
    \end{equation}
    is zero unless all the following conditions hold:
    \begin{enumerate}
        \item $n=2m$ for some $m\geq1$ and the sequence $\epsilon(1),\ldots,\epsilon(2m)$ is alternating;
        \item every $b_i$ is balanced;
        \item the word associated with $V_N^{\epsilon(1)}b_{1,N}\cdots V_N^{\epsilon(2m)}b_{2m,N}$ represents the identity of $G$.
    \end{enumerate}
    If these three conditions hold, then 
    \begin{equation}\label{eq:target-value}
        \lim_{N\to\infty}\E\left[\frac{1}{N} \mathrm{Tr}\left(\kappa_r^{\B_N}\bigl(V_N^{\epsilon(1)}b_{1,N},\ldots,V_N^{\epsilon(n)}b_{n,N}\bigr)\right)\right]=(-1)^{m-1}C_{m-1}.
    \end{equation}
\end{proposition}

The proof of this proposition uses two lemmas. The first transfers the limit of finite operator-valued cumulants to cumulants in $(\A_0,\B_0,E_L)$. The second computes the $\B_0$-valued cumulants.

\begin{lemma}[Convergence of scalarized cumulants]\label{lem:cumulant-convergence}
    Let $a_1,\ldots,a_n\in\A$ be cactus field-type monomials and let $a_{i,N}$ be their evaluations at $(V_N,V_N^*,\mathbf D_N)$. Then, we have
    \begin{equation}\label{eq:cumulant-convergence}
        \lim_{N\to\infty}\E\left[\frac1N\mathrm{Tr}\left(\kappa_n^{\B_N}(a_{1,N},\ldots,a_{n,N})\right)\right]=\tau_G\left(\kappa_n^{\B_0}(\Theta(a_1),\ldots,\Theta(a_n))\right).
    \end{equation}
\end{lemma}

\begin{proof}
    Since evaluation of graph monomials commutes with multiplication and with the diagonal operation, for any $\pi\in\mathrm{NC}(n)$, and any $a_1,\ldots,a_n\in\A$, we have
    \begin{equation}\label{eq: a prouver par induction}
        (\Delta_\pi(a_1,\ldots,a_n))(V_N,V_N^*,\mathbf{D}_N)=\Delta_\pi(a_{1,N},\cdots,a_{n,N}),
    \end{equation}
    where the first $\Delta$ is the operation on graph monomials and the second one is the diagonal of matrices. This equality justifies that we used the same notation. One can check by induction on the number of blocks of $\pi$ that 
    $$\Theta(\Delta_\pi(a_1,\cdots,a_n))=(E_L)_\pi(\Theta(a_1),\cdots,\Theta(a_n)).$$
    The case $\pi=1_n$ follows by multiplicativity of $\Theta$ and the fact that $\Theta\circ\Delta = E_L\circ\Theta$ and one uses the recursive definition of $\Delta_\pi$ and $E_\pi$ together with the multiplicativity of $\Theta$ to conclude. Using Equation \eqref{eq: state identification} together with \eqref{eq: a prouver par induction} leads to 
    $$\lim_{N\rightarrow\infty}\E\left[\frac1N \mathrm{Tr}(\Delta_\pi(a_{1,N},\cdots,a_{n,N}))\right]=\tau_G((E_L)_\pi(\Theta(a_1),\cdots,\Theta(a_n))).$$ Using Möbius inversion at finite $N$, one can take the limit term by term when $n$ is fixed to conclude the proof.
\end{proof}

Let $$u:=\lambda_G(e_L,1)\in \mathcal{A}_0,\qquad \alpha(b):=ubu^*\in\mathcal{B}_0,\;b\in \B_0.$$
Equivalently, $u=\Theta(v)$, where $v$ represents the graph monomial having a single edge labeled $v$ from input to output. Then $\alpha$ is the automorphism of $\B_0$ induced by the shift on $L$, and $u^n b u^{-n}=\alpha^n(b)$ for $b\in\B_0$ and $n\in\mathbb Z$.

Let $z$ be a scalar Haar unitary. Its scalar free cumulants vanish except for alternating words of even length, and
\begin{equation}\label{eq:haar-cumulants}
    \kappa_{2m}^{\mathbb C}(z,z^*,\ldots,z,z^*)=\kappa_{2m}^{\mathbb C}(z^*,z,\ldots,z^*,z)=(-1)^{m-1}C_{m-1},
\end{equation}
where $C_k$ is the $k$-th Catalan number.
This is the standard Haar-unitary cumulant formula (see \cite[Proposition~15.1]{nica_lectures_2006}).

\begin{lemma}[The $\B_0$-valued cumulants of $u$]\label{lem:shift-cumulants}
Let $\epsilon_1,\ldots,\epsilon_n\in\{1,-1\}$ and $b_1,\ldots,b_n\in\B_0$. Put $p_j=\epsilon_1+\cdots+\epsilon_j$. Then
\begin{equation}\label{eq:shift-cumulants}
    \kappa_n^{\B_0}(u^{\epsilon_1}b_1,\ldots,u^{\epsilon_n}b_n)=\kappa_n^{\mathbb C}(z^{\epsilon_1},\ldots,z^{\epsilon_n})\prod_{j=1}^n\alpha^{p_j}(b_j).
\end{equation}
\end{lemma}

\begin{proof}
    We first compute the moments. The relation $u^kb=\alpha^k(b)u^k$ gives, 
    \begin{equation}\label{eq: reorder}
        u^{\epsilon_1}b_1\cdots u^{\epsilon_n}b_n=\left(\prod_{j=1}^n\alpha^{p_j}(b_j)\right)u^{p_n}.
    \end{equation}
    Since $E_L(bu^k)=0$ for $k\neq0$ and $E_L(b)=b$ for $b\in\B_0$, we obtain
    $$E_L(u^{\epsilon_1}b_1\cdots u^{\epsilon_n}b_n)=\1_{\{p_n=0\}}\prod_{j=1}^n\alpha^{p_j}(b_j).$$ 
    We now show that the right-hand side of \eqref{eq:shift-cumulants} satisfies the operator-valued moment-cumulant formula. Let 
    $$K_n(u^{\epsilon_1}b_1,\cdots,u^{\epsilon_n}b_n):=\kappa_n^{\mathbb C}(z^{\epsilon_1},\ldots,z^{\epsilon_n})\prod_{j=1}^n\alpha^{p_j}(b_j),$$
    and, for $\pi\in NC(n)$, define $K_\pi$ from $K_n$ with the nested-formula of operator-valued cumulants.  Let us show that 
    \begin{equation}\label{eq: nested-cumulant}
        K_\pi(u^{\epsilon_1}b_1,\cdots,u^{\epsilon_n}b_n)=\kappa_\pi^\C(z^{\epsilon_1},\cdots,z^{\epsilon_n})\prod_{j=1}^{n}\alpha^{p_j}(b_j).
    \end{equation}
    If $\pi=1_n$, this is by definition. Assume $\pi$ has at least 2 blocks and let $V=\{q+1,\cdots,q+l\}$ be an interval of $\pi$. Let $r_j=\epsilon_{q+1}+\cdots+\epsilon_j$ for $q+1\leq j \leq q+l$. We thus have
    \begin{align*}
        K_l(u^{\epsilon_{q+1}}b_{q+1},\cdots,u^{\epsilon_{q+l}}b_{q+l})&=\kappa_l^\C(z^{\epsilon_{q+1}},\cdots,z^{\epsilon_{q+l}})\prod_{j=q+1}^{q+l}\alpha^{r_j}(b_j),
    \end{align*}
    which leads to 
    \begin{align*}
        K_\pi(u^{\epsilon_1}b_1,\cdots,u^{\epsilon_n}b_n)=\kappa_l^{\C}(&z^{\epsilon_{q+1}},\cdots,z^{\epsilon_{q+l}})\\
        &K_{\pi\setminus V}\left(u^{\epsilon_1}b_1,\dots,u^{\epsilon_q}b_q\prod_{j=q+1}^{q+l}\alpha^{r_j}(b_j),u^{\epsilon_{l+q+1}}b_{l+q+1},\cdots,u^{\epsilon_n}b_n\right).
    \end{align*}
    We apply the induction hypothesis with $\pi'=\pi\setminus V$ and $$b'_1=b_1,\cdots,b'_{q-1}=b_{q-1},\quad b'_q=b_q\prod_{j=q+1}^{q+l}\alpha^{r_j}(b_j),\quad b'_{q+1}=b_{l+q+1},\cdots,b'_{n-l}=b_n,$$
    to obtain 
    $$K_\pi(u^{\epsilon_1}b_1,\cdots,u^{\epsilon_n}b_n)=\kappa_\pi^{\C}(z^{\epsilon_1},\cdots,z^{\epsilon_n})\prod_{j=1}^{q-1}\alpha^{p_j}(b_j)\alpha^{p_q}(b'_q)\prod_{j=l+q+1}^{n}\alpha^{p_j-r_{q+l}}(b_j).$$
    Since $\alpha$ is a $*$-homomorphism and $p_q+r_j=p_j$ by definition, we have that 
    \begin{align*}
        \alpha^{p_q}(b'_q)=\alpha^{p_q}(b_q)\prod_{j=q+1}^{q+l}\alpha^{p_q+r_j}(b_j)=\prod_{j=q}^{q+l}\alpha^{p_j}(b_j).
    \end{align*}
    If $r_{q+l}=0$, it concludes the proof of Equation \eqref{eq: nested-cumulant}. On the other hand, if $r_{q+l}\neq 0$, either $l$ is odd and $\kappa_l^{\C}(\cdots)=0$, either $l$ is even and the signs  $\epsilon_{q+1},\dots,\epsilon_{q+l}$ cannot be alternated so $\kappa_l^{\C}(z^{\epsilon_{q+1}},\dots,z^{\epsilon_{q+l}})=0$ but so does $K_l(u^{\epsilon_{q+1}}b_{q+1},\dots,u^{\epsilon_{q+l}}b_{q+l})$. Hence both $K_\pi$ and $\kappa_\pi^\C$ are equal to zero which concludes the proof of Equation \eqref{eq: nested-cumulant}. 

    Let us now sum Equation \eqref{eq: nested-cumulant} over all non-crossing partitions $\pi$, 
    \begin{align*}
        \sum_{\pi\in NC(n)}K_\pi(u^{\epsilon_1}b_1,\ldots,u^{\epsilon_n}b_n)&=\left(\sum_{\pi\in NC(n)}\kappa_\pi^{\mathbb C}(z^{\epsilon_1},\ldots,z^{\epsilon_n})\right)\prod_{j=1}^n\alpha^{p_j}(b_j)\\
        &=\varphi (z^{\epsilon_1}\cdots z^{\epsilon_n})\prod_{j=1}^n\alpha^{p_j}(b_j)\\
        &=\1_{\{p_n=0\}}\prod_{j=1}^n\alpha^{p_j}(b_j),\\
        &=E_L(u^{\epsilon_1}b_1\cdots u^{\epsilon_n}b_n),
    \end{align*}    
    where $\varphi$ is the state of the Haar random variable $z$, it thus verifies $\varphi(z^n)=\1_{n=0}$. Since operator-valued cumulants are characterized by the moment-cumulant relation, it concludes the proof.
\end{proof}


\begin{proof}[Proof of Proposition \ref{prop: asymptotic op val cumulants}]
    Let $a_i:=v^{\epsilon_i}b_i$ denote the graph monomial having an edge labeled $v^{\epsilon_i}$ and a loop $b_i$ attached to it (with the convention that if $\epsilon_i={-1}$, the edge is labeled $v^{*}$). By Lemma \ref{lem:cumulant-convergence}, we have 
    \begin{equation}\label{eq: first eq proof prop}
        \lim_{N\to\infty}\E\left[\frac{1}{N} \mathrm{Tr}\left(\kappa_r^{\B_N}\bigl(V_N^{\epsilon(1)}b_{1,N},\ldots,V_N^{\epsilon(n)}b_{n,N}\bigr)\right)\right]=\tau_G(\kappa_n^{\B_0}(\Theta(a_1),\cdots,\Theta(a_n))).
    \end{equation}

    First assume that one of the cacti, say $b_i$ is not balanced. Then $\Theta(b_i)=0$, thus $\Theta(a_i)=0$. Hence the right-hand side of \eqref{eq: first eq proof prop} is zero by linearity of $\kappa_n^{\B_0}$.

    From now on, assume that every $b_i$ is balanced. We saw, in the beginning of this section that for all $i\in[n]$, there exists $f_i:\Z\rightarrow\Lambda$ with finite support such that $\Theta(b_i)=\lambda_L(f_i)\in\B_0$. Hence, for all $i\in[n]$, we have $\Theta(a_i)=u^{\epsilon_i}\lambda_L(f_i)$. By Lemma \ref{lem:shift-cumulants}, keeping the same notations, we have
    \begin{equation*}
        \kappa_n^{\B_0}(\Theta(a_1),\cdots,\Theta(a_n))=\kappa_n^{\C}(z^{\epsilon_1},\cdots,z^{\epsilon_n})\prod_{j=1}^{n}\alpha^{p_j}(\lambda_L(f_j)).
    \end{equation*}
    If $n$ is odd or if the $\epsilon_i$'s do not alternate, the scalar cumulant on the right-hand side above is zero. 

    From now on, we assume $n=2m$ and the $\epsilon_i$'s are alternating. Hence, in that case, we have $p_2m=0$ and 
    \begin{equation}\label{eq: cumulants preuve prop}
        \kappa_n^{\B_0}(\Theta(a_1),\cdots,\Theta(a_n))=(-1)^{m-1}C_{m-1}\prod_{j=1}^{2m}\alpha^{p_j}(\lambda_L(f_j)).
    \end{equation}
    Note that for $p\in\Z$ and $f\in L$, we have
    \begin{align*}
        \alpha^p(\lambda_L(f))=u^p\lambda_L(f)u^{-p}=\lambda_G((e_L,p)(f,0)(e_L,-p))=\lambda_G(\sigma_p(f),0)=\lambda_L(\sigma_p(f)),
    \end{align*}
    where $\sigma_p(f):k\in\Z\mapsto f(k-p)\in\Lambda$ is the shift. Hence, the product on the right-hand side of \eqref{eq: cumulants preuve prop} rewrites as 
    $$\prod_{j=1}^{2m}\alpha^{p_j}(\lambda_L(f_j))=\lambda_L\left(\prod_{j=1}^{2m}\sigma_{p_j}(f_j)\right),$$
    where the product is taken in order from $j=1$ to $n$. Since $\tau_G$ extracts the coefficient of $e_G$, we have
    $$\tau_G(\kappa_n^{\B_0}(\Theta(a_1),\dots,\Theta(a_{2m})))=(-1)^{m-1}C_{m-1}\1_{\prod\sigma_{p_j}(f_j)=e_L},$$
    so it remains to prove that $\prod_{j=1}^{2m}\sigma_{p_j}(f_j)=e_L$ if and only if $w(v^{\epsilon_1}b_1\cdots v^{\epsilon_{2m}}b_{2m})=e_G$. However, this is easily seen when we reorder the product similarly to \eqref{eq: reorder}, since $p_{2m}=0$ we have 
    \begin{align*}
        \prod_{j=1}^{2m}\sigma_{p_j}(f_j)=e_L&\Longleftrightarrow u^{\epsilon_1}\lambda_L(f_1)\cdots u^{\epsilon_{2m}}\lambda_L(f_{2m})=\lambda_G(e_G),\\
        &\Longleftrightarrow \Theta(v^{\epsilon_1}b_1\cdots v^{\epsilon_{2m}}b_{2m})=e_G.
    \end{align*}
    Since each $b_j$ is balanced, it concludes the proof.
\end{proof}

\subsection{The $R$-transform of $u+u^*$}

We define formally the $R$-transform as the formal power series
$$ R_x^{\B_0}(b):=\sum_{n\geq0}\kappa_{n+1}^{\B_0}(xb,\ldots,xb,x),$$
see \cite{dykema_multilinear_2005} or
\cite[Chapter~10]{mingo_free_2017}.

\begin{proposition}\label{analytic R transform}
Let $x=u+u^*$.  For $\|b\|<1/2$,
\begin{align*}
 R_x^{\B_0}(b)&=\sum_{m\geq0}(-1)^mC_m\left((\alpha(b)b)^m\alpha(b)+(\alpha^{-1}(b)b)^m\alpha^{-1}(b)\right)\\
 &=2\alpha(b)\bigl(1+\sqrt{1+4b\alpha(b)}\bigr)^{-1}+2\alpha^{-1}(b)\bigl(1+\sqrt{1+4b\alpha^{-1}(b)}\bigr)^{-1}.
\end{align*}
The square roots are the principal square roots given by functional calculus. 
\end{proposition}

\begin{proof}
By Proposition \ref{prop: asymptotic op val cumulants}, only the two alternating sign patterns in each even order survive giving the first equation. Since $C_m\leq4^m$ and $\alpha$ is an isometry, each summand is bounded by $4^m\|b\|^{2m+1}$; hence the series converges absolutely for
$\|b\|<1/2$.

The Catalan generating function is given, for $|z|<1/4,$ by
$$ \sum_{m\geq0}(-1)^m C_m z^m =\frac{2}{1+\sqrt{1+4z}}.$$
Applying it in the commutative $C^*$-algebra $\B_0$ gives the second equation. Note that $\B_0$ is commutative since $\Lambda$ is abelian. The principal square root exists because $\|4b\alpha^{\pm1}(b)\|<1$.
\end{proof}

    \section{Second-order distribution}

Once we have constructed the limiting operator-valued $C^{*}$-probability
space $\left(\A_{0},\B_{0},E_{L}\right)=\left(C_{r}^{*}(G),C_{r}^{*}(L),E_{L}\right),$we can express the limiting covariance between two graph monomials in terms of the first order distribution of their product, up to a symmetrization over the position of the lamplighter. In this subsection, we remain in the case $\Gamma=\Z$. We recall from the introduction that an element of $G$ is a couple $(f,g)$, where $f:\Gamma\rightarrow\Lambda$ has finite support and $g\in\Gamma$. We denote by $e_{L}$ the constant function equal to the identity of $\Lambda$ and by $e_{\Gamma}=0$ the identity of $\Gamma=\Z$. In the rest of this subsection, we use the additive notation on $\Gamma$, so that $gg'=g+g'$ and $g^{-1}=-g$. For clarity, in the following, for $(f,g)\in G$, we denote
\begin{equation*}
    U_{f,g}:=\lambda_{G}(f,g)\in C_{r}^{*}(G).
\end{equation*}
It follows from the product rule that
\begin{equation*}
    U_{f,g}U_{f',g'}=U_{f\cdot(f'\circ g^{-1}),g+g'},\qquad U_{f,g}^{*}=U_{f^{-1}\circ g,-g}.
\end{equation*}
The linear span of the family $(U_{f,g})_{(f,g)\in G}$ is the group algebra
$\C[G]$, which is dense in $\A_{0}$. Moreover, these elements are orthogonal for $\tau_G$ : $\tau_{G}\left(U_{f,g}^{*}U_{f',g'}\right) =\1_{(f,g)=(f',g')}$. Hence, the same family is an orthonormal basis of $L^{2}(C_{r}^{*}(G),\tau_{G})$. In the following, we only consider finite linear combinations of these unitaries.

For each couple $(f,g)$, we choose the canonical graph monomial obtained as follows. We consider the smallest interval of $\Z$ containing $\{0,g\}\cup\supp(f)$, draw the directed line on this interval, take $0$ as input and $g$ as output, and attach at every vertex $k$ the loops corresponding to $f(k)$. Its image in $C_{r}^{*}(G)$ through $\Theta$ is indeed $U_{f,g}$. If $x\in\C[G]$, we denote by $x_{N}$ the corresponding finite linear combination of graph monomials evaluated at the random matrices of the model.

\begin{definition}\label{def:first-second-functionals-fg}
    Let $x,y\in\C[G]$. The first and second order distributions are defined by
    \begin{align*}
        \Phi^{(1)}(x) &:=\lim_{N\rightarrow\infty}\frac{1}{N}\E\left[\mathrm{Tr}(x_{N})\right],\\
        \Phi^{(2)}(x,y)&:=\lim_{N\rightarrow\infty}\frac{1}{N}\E\left[\left(\mathrm{Tr}(x_{N})-\E[\mathrm{Tr}(x_{N})]\right)\left(\mathrm{Tr}(y_{N})-\E[\mathrm{Tr}(y_{N})]\right)\right].
    \end{align*}
\end{definition}

\begin{remark}\label{remark:first-second-fg}
    \begin{itemize}
        \item The existence of the above limits follows from the first order computation Section 2 and from the central limit theorem of Section 3.
        \item The first order distribution is the canonical trace of the reduced group $C^{*}$-algebra. More precisely, for all $x\in\C[G]$,
        \begin{equation}\label{eq:Phi1-trace-expectation-fg}
            \Phi^{(1)}(x)=\tau_{G}(x)=\tau_{G}(E_{L}(x)).
        \end{equation}
    \end{itemize}
\end{remark}

Note that $\Phi^{(2)}$ is a priori well-defined of graph polynomials but not on elements of $\C[G]$. Let us define $\Tilde{\Phi}^{(2)}$ the version of $\Phi^{(2)}$ defined on graph polynomial. It could be possible that a graph polynomial $a\in\A$ verifies $\Theta(a)=0$ but $\Tilde{\Phi}^{(2)}(a,b)\neq 0$ for some $b\in\A$. However, it is never the case. Indeed, first decompose $a=a_1+a_2$ where $a_1$ is a linear combination of balanced graph monomials and $a_2$ is a linear combination of non-balanced graph monomials. Thus $0=\Theta(a)=\Theta(a_1)$ since $\Theta(a_2)=0$ by definition. Note that this implies that $a_{1,N}:=a_1(V_N,\mathbf{D}_N)=0$ which is stronger than $a_1=0$.  Indeed, for a balanced cactus $c$, its evaluation at finite $N$ reduces to take the product of the matrices one encounters as one reads the labels around the contour of the cactus. Since $\Lambda$ is abelian, if two planted cycles are attached to another one through the same vertex, it does not matter which one we visit first. Hence, $a_N=a_{2,N}$. However, since $a_2$ is not balanced, the variance of $a_{2,N}$ is at most of order 1 (see the proof of Proposition \ref{TCL}). This shows by Cauchy Schwarz that the covariance between $a$ and $b$ is at most of order $N^{-1/2}$ so that $\Tilde{\Phi}^{(2)}(a,b)=0$ for any graph monomial $b$. Hence, $\Theta(a)=\Theta(a')$ implies $\Tilde{\Phi}^{(2)}(a,b)=\Tilde{\Phi}^{(2)}(a',b)$ so that $\Phi^{(2)}$ is well defined on $\A/\mathrm{Ker}\Theta\cong \C[G]$. 

We can now compute the first and second order distributions of the family $(U_{f,g})_{(f,g)\in G}$.

\begin{proposition}\label{prop:Phi1-Phi2-fg}
    Let $f,f'\in L$ and $g,g'\in\Gamma$. We have
    \begin{equation}\label{eq:Phi1-U-fg}
        \Phi^{(1)}(U_{f,g})=\tau_{G}(U_{f,g})=\1_{f=e_{L},\;g=e_{\Gamma}}.
    \end{equation}
    Moreover, if $g\neq e_{\Gamma}$ or $g'\neq e_{\Gamma}$,
    \begin{equation}\label{eq:Phi2-non-closed-fg}
        \Phi^{(2)}(U_{f,g},U_{f',g'})=0 .
    \end{equation}
    Finally, when $g=g'=e_{\Gamma}$,
    \begin{align}
        \Phi^{(2)}(U_{f,e_{\Gamma}},U_{f',e_{\Gamma}})
        &=\1_{f\neq e_{L}}\1_{f'\neq e_{L}}
        \sum_{h\in\Gamma}
        \Phi^{(1)}\left(U_{f,h}U_{f',h^{-1}}\right)
        \label{eq:Phi2-first-order-fg}\\
        &=\1_{f\neq e_{L}}\1_{f'\neq e_{L}}
        \#\left\{h\in\Gamma:\
        f\cdot(f'\circ h^{-1})=e_{L}\right\}.
        \label{eq:Phi2-count-fg}
    \end{align}
\end{proposition}

\begin{proof}
    Equation \eqref{eq:Phi1-U-fg} follows directly from Equation \eqref{eq:Phi1-trace-expectation-fg} and from the definition of the canonical trace on $C_{r}^{*}(G)$.

    Let us now consider the second order distribution. If $g\neq e_{\Gamma}$, identifying the input and the output of the canonical graph associated to $U_{f,g}$ creates an unbalanced cycle. Such a graph does not contribute to the limiting covariance computed in Section 3. The same argument applies when $g'\neq e_{\Gamma}$, which proves Equation \eqref{eq:Phi2-non-closed-fg}.

    We now assume that $g=g'=e_{\Gamma}$. If $f=e_{L}$, the corresponding trace is deterministic and its centered version is zero. The same holds when $f'=e_{L}$. We can therefore assume that both lamp configurations are non-trivial. A non-zero term in the covariance is obtained by gluing the two closed graph monomials along their directed lines. Once a vertex of the first line is identified with a vertex of the second one, all the remaining identifications are forced. Such a gluing is consequently determined by an element $h\in\Gamma$, which describes the relative position of the two lamp configurations. The resulting graph contributes
    if and only if $f\cdot(f'\circ h^{-1})=e_{L},$ which concludes the proof.
\end{proof}

Equation \eqref{eq:Phi2-count-fg} has the following simple interpretation : the second order distribution counts whether the lamp configurations can be obtained from one another by a simple translation along $\Gamma$. We now use this observation to recover the second order distribution from the first order one after averaging over the position of the lamplighter.

\begin{definition}\label{def:En-fg}
    For $f\in L$, we denote $\diam(f):=\inf\left\{n\in\N:\\supp(f)\subset\llbracket -n,n\rrbracket\right\}$, and $\diam(e_{L})=0.$ For $n\geq0$, we define $\mathcal E_{n}:=\operatorname{span}\left\{U_{f,g}:\\diam(f)\leq n,\ |g|\leq n\right\}.$ We then have
    \begin{equation*}
        \bigcup_{n\geq0}\mathcal E_{n}=\C[G],\qquad\overline{\bigcup_{n\geq0}\mathcal E_{n}}^{\|\cdot\|}=\A_{0}.
    \end{equation*}
\end{definition}

We have now all the ingredients to state the relation between the first and second order distributions.

\begin{proposition}\label{prop: second-order distrib}
    Let $m,n\geq0$, $x\in\mathcal E_{n}$ and $y\in\mathcal E_{m}$. Assume that for all $k\in\Gamma=\Z$,
    \begin{equation}\label{eq:no-pure-shift}
        \Phi^{(1)}\left(xU_{e_{L},k}\right)=\Phi^{(1)}\left(yU_{e_{L},k}\right)=0.
    \end{equation}
    For $k\in\Z$, we define $x^{(k)}:=xU_{e_{L},-k}=xu^{-k}$ and $ y^{(k)}:=yU_{e_{L},-k}=yu^{-k}.$ For $R,S\geq0$, let
    \begin{equation*}
        X^{(R)}:=\sum_{k=-R}^{R}x^{(k)},\qquad Y^{(S)}:=\sum_{k=-S}^{S}y^{(k)}.
    \end{equation*}
    If $R\geq2n+m,$ and $S\geq n+2m,$ then
    \begin{equation}\label{eq:first-second-symmetrization}
        \Phi^{(2)}\left(X^{(R)},Y^{(S)}\right)=\Phi^{(1)}\left(X^{(R)}Y^{(S)}\right).
    \end{equation}
\end{proposition}

\begin{remark}\label{rem:remove-pure-shifts}
    The assumption in Equation \eqref{eq:no-pure-shift} is not restrictive. Indeed, if $x\in\mathcal E_{n}$, we can replace it by
    \begin{equation*}
        \overset{\circ}{x}:=x-\sum_{k=-n}^{n}
        \Phi^{(1)}\left(xU_{e_{L},-k}\right)U_{e_{L},k}.
    \end{equation*}
    Then $\overset{\circ}{x}\in\mathcal E_{n}$ and $\Phi^{(1)}(\overset{\circ}{x}U_{e_{L},k})=0$ for every $k\in\Z$. The same construction can be done for $y$.
\end{remark}

\begin{proof}
    We first write $x$ and $y$ in the family introduced above :
    \begin{equation*}
        x=\sum_{\substack{\diam(f)\leq n\\ |g|\leq n}}\lambda_{f}^{(g)}U_{f,g},\qquad y=\sum_{\substack{\diam(f')\leq m\\ |g'|\leq m}}\mu_{f'}^{(g')}U_{f',g'}.
    \end{equation*}
    The assumption in Equation \eqref{eq:no-pure-shift} means exactly that $\lambda_{e_{L}}^{(g)}=0,$ and $\mu_{e_{L}}^{(g')}=0,$ for every $g,g'\in\Z$. We set
    \begin{equation*}
        \Lambda_{f}:=\sum_{|g|\leq n}\lambda_{f}^{(g)},\qquad M_{f'}:=\sum_{|g'|\leq m}\mu_{f'}^{(g')}.
    \end{equation*}

    Since $R\geq n$, the coefficient of $U_{f,e_{\Gamma}}$ in $X^{(R)}$ is $\Lambda_{f}$. Similarly, the coefficient of $U_{f',e_{\Gamma}}$ in $Y^{(S)}$ is $M_{f'}$. By Proposition \ref{prop:Phi1-Phi2-fg} and by the bilinearity of $\Phi^{(2)}$, we obtain
    \begin{equation}\label{eq:Phi2-expanded-fg}
        \Phi^{(2)}\left(X^{(R)},Y^{(S)}\right) =\sum_{f\neq e_{L}}\sum_{f'\neq e_{L}}\Lambda_{f}M_{f'}\sum_{h\in\Gamma}\Phi^{(1)}\left(U_{f,h}U_{f',-h}\right).
    \end{equation}

    Let us now compute the right-hand side of Equation \eqref{eq:first-second-symmetrization}. A typical term in the product $X^{(R)}Y^{(S)}$ is $U_{f,g-k}U_{f',g'-\ell},$ where $|g|\leq n$, $|g'|\leq m$, $|k|\leq R$ and $|\ell|\leq S$. Assume that its first order contribution is non zero. By the product rule, we have $g-k+g'-\ell=0,$ and $f\cdot\left(f'\circ(g-k)^{-1}\right)=e_{L}.$ Let $h=g-k$, then $g'-l=-h$. Moreover, if $f$ or $f'$ is trivial, then we are in the case of Equation \eqref{eq:no-pure-shift} and the first order contribution of this term is zero. Furthermore, if $f,f'$ are non-trivial, the support conditions $\diam(f)\leq n$ and $\diam(f')\leq m$ imply $|h|\leq n+m.$ Indeed, the support of $f$ is the translate by $h$ of the support of $f'$. Hence $|k|=|g-h|\leq2n+m\leq R$, and $|\ell|=|g'+h|\leq n+2m\leq S.$ Hence, summing over $g,g',f$ and $f'$, gives
    \begin{equation*}
        \Phi^{(1)}\left(X^{(R)}Y^{(S)}\right)=\sum_{f\neq e_{L}}\sum_{f'\neq e_{L}}\Lambda_{f}M_{f'}\sum_{h\in\Gamma}\Phi^{(1)}\left(U_{f,h}U_{f',h^{-1}}\right),
    \end{equation*}
    which proves the proposition.
\end{proof}

Proposition \ref{prop: second-order distrib} shows that, after averaging over a sufficiently large set of positions of the lamplighter, the limiting covariance is recovered from the first order trace on $C_{r}^{*}(G)$. This is the operator-algebraic counterpart of the gluing mechanism which appeared in the proof of the central limit theorem.

    \section{Acknowledgments}

The core idea of this paper, namely the random matrix model of the lamplighter group, is originally from Benson Au. The author is thankful to him for sharing this idea and many others. The author would also like to thank Camille Male for numerous reviews and advice on this article.

    \bibliography{biblio_lamplighter.bib}

@article{DicksSchick2002,
  author = {Warren Dicks and Thomas Schick},
  title = {The spectral measure of certain elements of the complex group ring of a wreath product},
  journal = {Geometriae Dedicata},
  year = {2002},
  volume = {93},
  pages = {121--137},
  doi = {10.1023/A:1020381532489}
}

@article{Revelle2003,
  author = {David Revelle},
  title = {Heat kernel asymptotics on the lamplighter group},
  journal = {Electronic Communications in Probability},
  year = {2003},
  volume = {8},
  pages = {142--151}
}

@article{KambitesSilvaSteinberg2005,
 author = {Kambites, Mark and Silva, Pedro V. and Steinberg, Benjamin},
 title = {The spectra of lamplighter groups and {Cayley} machines.},
 fjournal = {Geometriae Dedicata},
 journal = {Geom. Dedicata},
 issn = {0046-5755},
 volume = {120},
 pages = {193--227},
 year = {2006},
 language = {English},
 doi = {10.1007/s10711-006-9086-8},
 zbMATH = {5053972},
 Zbl = {1168.20012}
}

@article{au_freeness_2021,
	title = {Freeness over the diagonal for large random matrices},
	volume = {49},
	issn = {0091-1798},
	doi = {10.1214/20-AOP1447},
	number = {1},
	journal = {The Annals of Probability},
	author = {Au, Benson and Cébron, Guillaume and Dahlqvist, Antoine and Gabriel, Franck and Male, Camille},
	year = {2021},
	pages = {157--179},
}

@book{nica_lectures_2006,
	series = {Lond. {Math}. {Soc}. {Lect}. {Note} {Ser}.},
	title = {Lectures on the combinatorics of free probability},
	volume = {335},
	isbn = {0-521-85852-6},
	url = {https://www.math.uni-sb.de/ag/speicher/publikationen/Nica-Speicher.pdf},
	publisher = {Cambridge: Cambridge University Press},
	author = {Nica, Alexandru and Speicher, Roland},
	year = {2006},
	doi = {10.1017/CBO9780511735127},
}

@book{mingo_free_2017,
	series = {Fields {Inst}. {Monogr}.},
	title = {Free probability and random matrices},
	volume = {35},
	isbn = {978-1-4939-6941-8 978-1-4939-6942-5},
	url = {https://mast.queensu.ca/~mingo/mingo_speicher_2017.pdf},
	publisher = {Toronto: The Fields Institute for Research in the Mathematical Sciences; New York, NY: Springer},
	author = {Mingo, James A. and Speicher, Roland},
	year = {2017},
	doi = {10.1007/978-1-4939-6942-5},
}

@book{male_traffic_2020,
	series = {Mem. {Am}. {Math}. {Soc}.},
	title = {Traffic distributions and independence: permutation invariant random matrices and the three notions of independence},
	volume = {1300},
	isbn = {978-1-4704-4298-9 978-1-4704-6399-1},
	url = {https://arxiv.org/pdf/1111.4662},
	publisher = {Providence, RI: American Mathematical Society (AMS)},
	author = {Male, Camille},
	year = {2020},
	doi = {10.1090/memo/1300},
}

@book{brown_ozawa_2008,
  author    = {Brown, Nathanial P. and Ozawa, Narutaka},
  title     = {{$C^*$}-Algebras and Finite-Dimensional Approximations},
  series    = {Graduate Studies in Mathematics},
  volume    = {88},
  publisher = {American Mathematical Society},
  address   = {Providence, RI},
  year      = {2008},
  isbn      = {978-0-8218-4381-9},
  doi       = {10.1090/gsm/088}
}

@book{paulsen_cb_2002,
  author    = {Paulsen, Vern I.},
  title     = {Completely Bounded Maps and Operator Algebras},
  series    = {Cambridge Studies in Advanced Mathematics},
  volume    = {78},
  publisher = {Cambridge University Press},
  address   = {Cambridge},
  year      = {2002},
  isbn      = {978-0-521-81669-4},
  doi       = {10.1017/CBO9780511546631}
}

@article{dykema_multilinear_2005,
    title = {Multilinear function series and transforms in free probability theory},
    journal = {Advances in Mathematics},
    volume = {208},
    number = {1},
    pages = {351-407},
    year = {2007},
    issn = {0001-8708},
    doi = {https://doi.org/10.1016/j.aim.2006.02.011},
    url = {https://www.sciencedirect.com/science/article/pii/S0001870806000703},
    author = {Kenneth J. Dykema}
}

@article{tomiyama_1957,
  author  = {Tomiyama, Jun},
  title   = {On the Projection of Norm One in {$W^*$}-Algebras},
  journal = {Proceedings of the Japan Academy},
  volume  = {33},
  number  = {10},
  pages   = {608--612},
  year    = {1957},
  doi     = {10.3792/pja/1195524885}
}

@article{male_limiting_2017,
	title = {The limiting distributions of large heavy {Wigner} and arbitrary random matrices},
	volume = {272},
	issn = {0022-1236},
	url = {https://www.sciencedirect.com/science/article/pii/S0022123616303044},
	doi = {https://doi.org/10.1016/j.jfa.2016.10.001},
	number = {1},
	journal = {Journal of Functional Analysis},
	author = {Male, Camille},
	year = {2017},
	pages = {1--46},
}

@misc{arras,
      title={Random Schr\"odinger operators and convolution on wreath products}, 
      author={Adam Arras},
      year={2025},
      note = {arXiv:math/2505.22485},
      eprint={2505.22485},
      archivePrefix={arXiv},
      primaryClass={math.PR},
      url={https://arxiv.org/abs/2505.22485}, 
}

@article{kesten_symmetric_1959,
	title = {Symmetric random walks on groups},
	volume = {92},
	issn = {0002-9947, 1088-6850},
	url = {https://www.ams.org/tran/1959-092-02/S0002-9947-1959-0109367-6/},
	doi = {10.1090/S0002-9947-1959-0109367-6},
	number = {2},
	urldate = {2024-10-14},
	journal = {Transactions of the American Mathematical Society},
	author = {Kesten, Harry},
	year = {1959},
	pages = {336--354},
}

@article{benaych-georges_central_2014,
	title = {Central {Limit} {Theorems} for {Linear} {Statistics} of {Heavy} {Tailed} {Random} {Matrices}},
	volume = {329},
	issn = {1432-0916},
	url = {https://doi.org/10.1007/s00220-014-1975-3},
	doi = {10.1007/s00220-014-1975-3},
	number = {2},
	journal = {Communications in Mathematical Physics},
	author = {Benaych-Georges, Florent and Guionnet, Alice and Male, Camille},
	month = jul,
	year = {2014},
	pages = {641--686},
}

@article{kammoun_small_2023,
    author = {Kammoun, Mohamed Slim and Ma{\"{\i}}da, Myl{\`e}ne},
    title = {A product of invariant random permutations has the same small cycle structure as uniform},
    fjournal = {Electronic Communications in Probability},
    journal = {Electron. Commun. Probab.},
    issn = {1083-589X},
    volume = {25},
    pages = {14},
    note = {Id/No 57},
    year = {2020},
    language = {English},
    doi = {10.1214/20-ECP334},
    zbMATH = {7252777},
    Zbl = {1469.60044}
}

@article{grigorchuk_lamplighter_2001,
	title = {The {Lamplighter} {Group} as a {Group} {Generated} by a 2-state {Automaton}, and its {Spectrum}},
	volume = {87},
	issn = {1572-9168},
	url = {https://doi.org/10.1023/A:1012061801279},
	doi = {10.1023/A:1012061801279},
	number = {1},
	journal = {Geometriae Dedicata},
	author = {Grigorchuk, Rostislav I. and Żuk, Andrzej},
	month = aug,
	year = {2001},
	pages = {209--244},
}

@article{mohar_survey_1989,
	title = {A survey on spectra of infinite graphs},
	volume = {21},
	issn = {0024-6093},
	doi = {10.1112/blms/21.3.209},
	language = {English},
	number = {3},
	journal = {Bulletin of the London Mathematical Society},
	author = {Mohar, Bojan and Woess, Wolfgang},
	year = {1989},
	pages = {209--234},
}

@article{grigorchuk_spectra_2019,
    author = {Grigorchuk, Rostislav and Simanek, Brian},
    year = {2020},
    month = {03},
    pages = {1},
    title = {Spectra of Cayley graphs of the lamplighter group and random Schrödinger operators},
    volume = {374},
    journal = {Transactions of the American Mathematical Society},
    doi = {10.1090/tran/8156}
}

@article{speicher_combinatorial_1998,
	title = {Combinatorial theory of the free product with amalgamation and operator-valued free probability theory},
	volume = {132},
	issn = {0065-9266, 1947-6221},
	url = {http://www.ams.org/memo/0627},
	doi = {10.1090/memo/0627},
	language = {en},
	number = {627},
	urldate = {2025-01-13},
	journal = {Memoirs of the American Mathematical Society},
	author = {Speicher, Roland},
	year = {1998},
	pages = {0--0},
}

@inproceedings{voiculescu_symmetries_1985,
	address = {Berlin, Heidelberg},
	title = {Symmetries of some reduced free product {C}*-algebras},
	isbn = {978-3-540-39514-0},
	doi = {10.1007/BFb0074909},
	language = {en},
	booktitle = {Operator {Algebras} and their {Connections} with {Topology} and {Ergodic} {Theory}},
	publisher = {Springer},
	author = {Voiculescu, Dan},
	editor = {Araki, Huzihiro and Moore, Calvin C. and Stratila, Şerban-Valentin and Voiculescu, Dan-Virgil},
	year = {1985},
	pages = {556--588},
}

@misc{accardi_decompositions_2006,
	title = {Decompositions of the free product of graphs},
	url = {http://arxiv.org/abs/math/0609329},
	doi = {10.48550/arXiv.math/0609329},
	urldate = {2025-01-16},
	publisher = {arXiv},
	author = {Accardi, Luigi and Lenczewski, Romuald and Salapata, Rafal},
	month = sep,
	year = {2006},
	note = {arXiv:math/0609329},
}

@article{speicher_universal_1997,
	title = {On universal products},
	volume = {12},
	doi = {10.1090/fic/012/12},
	author = {Speicher, Roland},
	month = jan,
	year = {1997},
	note = {ISBN: 9780821806753},
	pages = {257--266},
    journal = {Free Probability Theory}
}

@misc{bordenave2024largedeviationsmacroscopicobservables,
      title={Large deviations for macroscopic observables of heavy-tailed matrices}, 
      author={Charles Bordenave and Alice Guionnet and Camille Male},
      year={2024},
      note = {arXiv:math/2409.14027},
      eprint={2409.14027},
      archivePrefix={arXiv},
      primaryClass={math.PR},
      url={https://arxiv.org/abs/2409.14027}, 
}
    \bibliographystyle{plain} 

\end{document}